\documentclass[11pt,a4paper,reqno]{article}

\usepackage{setspace}
\usepackage{graphicx}
\usepackage{amssymb}
\usepackage{epstopdf}
\usepackage{tikz}
\usepackage{latexsym}

\usepackage{amsmath,amsthm}

\usepackage[mathscr]{eucal}
\usepackage[all]{xy}
\usepackage{mathrsfs}
\usepackage{authblk}
\usepackage{hyperref}

\usepackage{bbm}
\usepackage{amsfonts}

\usepackage{stmaryrd}

\newtheorem{theorem}{Theorem}[section]
\newtheorem{lemma}[theorem]{Lemma}
\newtheorem{definition}[theorem]{Definition}

\newtheorem{proposition}[theorem]{Proposition}

\catcode`\@=11

\@addtoreset{equation}{section}

\theoremstyle{plain}
\newtheorem{remark}[theorem]{Remark}
\newtheorem{example}[theorem]{Example}

\newtheorem{thm}{Theorem}

\newcommand{\bb}{ \mathbbm{b}}

\newcommand{\R}{\mathbb{R}}
\newcommand{\C}{\mathbb{C}}
\newcommand{\Z}{\mathbb{Z}}

\begin{document}

	\title{ On the Maslov-type index for general paths of symplectic matrices}

	\author{Hai-Long Her\footnote{ Department of Mathematics, Jinan University, Guangzhou, China. E-mail: hailongher@jnu.edu.cn} \ \ and\ \  \  
	Qiyu Zhong\footnote{ School of Mathematics, Sichuan University, Chengdu, China. E-mail: zhongqiyu@stu.scu.edu.cn} }
	
	%\thanks{Both authors are partially supported by  the project No. 11671209 of NSFC and the project No. 2021A1515010379 of Guangdong Basic and Applied Basic Research Foundation, China.}}

\renewcommand\Affilfont{\small}

%\affil[1]{Department of Mathematics, Jinan University, Guangzhou 510632, China}
%\affil[2]{School of Mathematics, Sichuan University, Chengdu 610065, China}

%\address{$^{\dagger}$ \quad Department of Mathematics, Jinan University, Guangzhou 510632, China} 

\date{}

\maketitle

\noindent {\small{\bf Abstract:}  In this article, we define an index of Maslov type for general symplectic paths which have two arbitrary end points. 
This Maslov-type index is a partial generalization of the Conley-Zehnder-Long index  in the sense that the degenerate set of symplectic matrices is larger.
The method of constructing the index is direct without taking advantage of Maslov index of Lagrangian paths and consistent no matter whether the starting point of the path is identity or not, which is different  from the ones for Long's Maslov-type index and Liu's $L_0$-index. Some natural properties for the index are verified. 
We review other versions of Maslov indices and compare them with our definition. 
In particular, this Maslov-type index can be regarded as a partial realization of Cappell-Lee-Miller index for a pair of Lagrangian paths from the point of view of index for symplectic paths.\\
	
	

	\tableofcontents
	
} %\small end	

\section{Introduction}

In 1965, an index was originally introduced by Maslov in \cite{M} for an oriented closed curve  in a Lagrangian submanifold, and  used to deal with the problem of asymptotic expression of the solution of the Schr\"odinger equations. In 1967, Arnold \cite{A1} accomplished the rigorous mathematical definition of Maslov index, which is defined as the index of a pair of Lagrangian loops, and explained it as the intersection number of a path of Lagrangian subspaces with the so-called Maslov cycle.
The method of defining Maslov index can be generalized to construct various Maslov-type indices.
In 1984, Conley and Zehnder\cite{CZ2} formulated a kind of Maslov-type index for paths of symplectic matrices starting at identity, called Conley-Zehnder index, which was constructed for the aim of studying non-degenerate periodic solutions of Hamiltonian systems. We remark that Conley-Zehnder index itself is important for the construction of Floer homology\cite{F1,F2} and is applied to deal with the problem of Arnold conjecture\cite{A0,A2,CZ1,CZ2,CZ3,F1,F2,SZ,FO,LT}. Furthermore,  Maslov-type indices for degenerate symplectic paths starting at identity were firstly constructed in a direct way by Long\cite{L 1990} and Viterbo\cite{V} in 1990. In 2007, Liu\cite{Liu1} also constructed Maslov-type index for symplectic paths with Lagrangian boundary conditions. Delicate iterative formulae for  Maslov-type indices were established by Long school\cite{L,Liu2}, which turn out to be extremely useful for investigating various problems relating to the periodic solutions of Hamiltonian systems, closed characteristics, brake orbits and closed geodesics, arising from dynamical systems, celestial mechanics, contact geometry, Riemannian geometry and Finsler geometry, etc.\cite{LLZ, LZ, LZZ, LiuZ, HLS, WHL, BL}.

In this paper, we study the problem of constructing  Maslov-type index for general symplectic paths in a consistent way no matter whether the starting point of the path is identity or not.
Note that Conley-Zehnder index is essentially a Maslov index defined  for those symplectic paths of which the starting point is the identity and the end point, which is a symplectic matrix, satisfies the non-degenerate condition that it has no eigenvalues equal to $1$. More formally, denote by
$$\mathcal{P}(2n,\mathbb{R}):=\{\Phi:[0,1]\to Sp(2n,\R) \ \text{is continuous} \}$$
the space of general paths of $2n\times 2n$ symplectic matrices, where $\Phi(0)$ and $\Phi(1)$ are arbitrary symplectic matrices. 
Then the Conley-Zehnder index can be constructed for a path $\Phi\in\mathcal{P}(2n,\mathbb{R})$ such that $\Phi(0)=I_{2n}$ and the determinant $\det\big(I_{2n}-\Phi(1)\big)\neq 0$.
Roughly speaking, the crucial  idea of the first step of constructing the Conley-Zehnder index is to establish a correspondence from the symplectic paths to 
the paths on the unit circle $\mathbb{S}^1\subset \C$ and then to get a rotation number counting the ratio of the rotation angle on unit circle quotient $\pi$. Thus, one should define a mapping as
\begin{equation}\label{1.1}
	\rho: Sp(V,\omega) \to \mathbb{S}^1
\end{equation}
satisfying some specific rules, where $Sp(V,\omega)$ is the symplectic group of the symplectic space $(V,\omega)$. Salamon and Zehnder \cite{SZ} continued the work of \cite{CZ2} and gave the rules to define this mapping, which shows that this mapping can be defined uniquely under the four designated properties. The precise definition and rules are restated in the Theorem \ref{Thm SZ} below, we refer the reader to \eqref{2.2} for the expression of this mapping. 
We also remark that there exists a subtle difference between Conley-Zehnder's and Salamon-Zehnder's definitions of the map \eqref{1.1} and their derived rotation numbers. See the Remark \ref{CZ rotation num} below for more discussion.

It is not satisfied that the Conley-Zehnder index is defined only for the non-degenerate symplectic paths. 
For example, even for the simplest constant path $\Phi(t)\equiv I_{2n}$, there was no an associated Conley-Zehnder index.
Long \cite{L 1990,L0} generalized the Conley-Zehnder index and considered the degenerate paths $\Phi\in\mathcal{P}(2n,\mathbb{R})$ (i.e. $\Phi(0)=I_{2n}$ 
and $\det\big(I_{2n}-\Phi(1)\big)= 0$). In particular, in \cite{L 1990} Long originally formulated the idea of rotational perturbation to deal with the degenerate paths.
Thus the degenerate paths can be deformed into the non-degenerate ones, for which the Conley-Zehnder index is well-defined. Then the Long's index is defined as the infimum of these Conley-Zehnder indices (see Definition \ref{Def L}). Long \cite{L 1999} also studied the topological structure of the degenerate set of an arbitrary eigenvalue $\omega$ on $\mathbb{S}^1$ rather than only $1$ and in \cite{L} defined the $\omega$-index for symplectic paths. Liu \cite{Liu1} defined the $L_0$-index (for symplectic paths with Lagrangian boundary conditions) via a fixed Lagrangian subspace $L_0$. We will review the works by Conley and Zehnder as well as by Long in section \ref{rev CZL index} and call it ``Conley-Zehnder-Long index". $L_0$-index will be reviewed in section \ref{rev L0 index}.  For the Conley-Zehnder-Long index of the degenerate paths, we may sometimes also call it ``Long index" for simplicity.
One can easily verify that the Long index of the simplest constant path $\Phi(t)\equiv I_{2n}$ is $-n\neq 0$, while it might be more reasonable that  some index of a constant path is  intuitively supposed to be zero.

Between 1972 and 1981, Leray \cite{Le} made important contributions to extending Arnold's work. He did not consider only loops but also any paths in the Lagrangian Grassmannian and defined a kind of Maslov index for a pair of Lagrangian paths with transversal conditions. In 1993, Robbin and Salamon \cite{RS1} continued to advance research on the index of Lagrange pairs. Their index for any Lagrangian pairs can indirectly induce an index for a general path of symplectic matrices, called ``generalized Conley-Zehnder index" or ``Robbin-Salamon index" for symplectic paths. While this kind of index by Robbin-Salamon is in general a half integer rather than an integer. 
In 2009, Gosson \cite{Go} recovered the Robbin-Salamon index and Conley-Zehnder index via Leray's index and led to a formula for the Conley–Zehnder index of a product of two paths.
In 2014, Gutt  gave an axiomatic characterization of Robbin-Salamon's  generalized Conley-Zehnder index, which is  based on Robbin-Salamon index for some induced Lagrangian pairs in the product symplectic space $\big(\R^{2n}\times\R^{2n}, \omega_0\times(-\omega_0)\big)$.  Moreover, a  formula of computing the Robbin-Salamon index is given in \cite{G}.

In 1994, in order to unify different definitions, Cappell, Lee and Miller \cite{CLM} originally formulated a system of axioms for the pairs of Lagrangian paths and introduced four definitions of Maslov indices for Lagrangian pairs. Moreover, they showed that these  definitions  satisfy this system of axioms so that they are equivalent to one another.
Note that this system  of axioms is not applicable to define Maslov index for symplectic paths, since the index for Lagrangian pairs is of symplectic invariance (see Property V of \cite{CLM}) while the index for symplectic paths does not have this property. 
%Otherwise, for any path of symplectic matrices, if we choose the path of the inverse matrices acting on it, then by symplectic invariance the index must be always equal to zero, which is unreasonable.
Since it is not so straightforward to follow this system of axioms for the index of the pairs of Lagrangian paths to construct the index for symplectic paths, we prefer to follow the rules formulated in \cite{SZ} to define the  Maslov-type index for more general symplectic paths.
%Perhaps we can also correspond symplectic paths to Lagrangian pairs and define it by the index of Lagrangian pairs, for example, Robbin and Salamon \cite{RS1} had done this work.
Note that the Cappell-Lee-Miller index for Lagrangian path pairs can also naturally but indirectly induce an index for general symplectic paths(see Definitions \ref{Def-CLM} and \ref{Def-CLM-sp}).
However, since those induced definitions of indices of symplectic paths are indirect, it is not quite clear how to calculate these induced indices. 
Then it is natural  to study the relationship between Cappell-Lee-Miller index and other versions of Maslov-type indices. 
It is just one of  motivations for this work that we want to understand the Cappell-Lee-Miller index from the point of view of index of symplectic paths.
For more applications and insight into the significance of the Cappell-Lee-Miller index, one can  see \cite{CM1, CM2, CM3}, 
and see also the preprint \cite{CM} for the recent work by Cappell and Miller on a comparison of Cappell-Lee-Miller index with the index defined in this paper.

To construct some index of general symplectic paths directly, or say symplectic path segments, Long \cite{L 1990,L0} and Liu\cite{Liu1} made important contributions, respectively. A symplectic path segment has two arbitrary end points and hence it is just a general symplectic path. Their constructions use the idea that a general symplectic path always corresponds to two symplectic paths starting at $I_{2n}$, then the index for general symplectic path can be defined as the difference of the indices of those two symplectic paths starting at $I_{2n}$ (see Definition \ref{SPS  index}). This is a concise method to deal with the general symplectic paths which is introduced by Long in \cite{L0} and in Definition 6.2.9 of \cite{L}. This index is called the index of symplectic path segment and we simply call it ``SPS index", which will be recalled in section \ref{rev SPS index}. Note that for the same symplectic path the Conley-Zehnder-Long index, Liu's  $L_0$-index and SPS index might not be the same (see Remark \ref{rk L}). In particular, the SPS index of constant path, $e.g.$ $\Phi(t)\equiv I_{2n}$, is zero. 

In this paper, we directly construct a Maslov-type index $\mu(\Phi)$ for a general symplectic path $\Phi \in \mathcal{P}(2n,\mathbb{R})$ as a partial generalization of the Conley-Zehnder index by using a different method  from ones of Long and Liu. The feature of our method of constructing Maslov-type index is that we try to deal with the general symplectic paths straightforwardly rather than first defining some index for paths starting from identity. Hence we can get a consistent construction of index for the general symplectic paths no matter whether they are starting from identity or not. This construction involves {\bf orthogonalization} at two ends (see \eqref{4.3}) and {\bf global perturbation} (see \eqref{4.8}) which are different from the previous ones. The aim of orthogonalization is to turn the two end points of a path into the orthogonal and symplectic matrices which have better properties (e.g. they are still orthogonal and symplectic under the global perturbation). The global perturbation can ensure that the two end points change into non-degenerate ones so that we can apply the {\bf extension} (see \eqref{4.13}) to the construction. See the precise Definition \ref{Def Maslov.type.index} below. If we just generalize the method used by Conley-Zehnder, it seems impossible to get an integer. For example, consider a general path
\begin{align}\label{1.2}
	\Phi(t) = \begin{pmatrix}
		\cos \pi (t+\frac{1}{2}) & -\sin \pi (t+\frac{1}{2}) \\
		\sin \pi (t+\frac{1}{2}) & \cos \pi (t+\frac{1}{2}) \\
	\end{pmatrix},\ 0 \leq t \leq 1.
\end{align}
Conley-Zehnder index can not directly apply to this symplectic path \eqref{1.2} since $\Phi(0)$ is not identity, 
while $\Phi$ has the generalized Maslov-type index $1$ with respect to the Definition \ref{Def Maslov.type.index} (see also Remark \ref{remark}).
In fact, our  Maslov-type index of constant path, $e.g.$ $\Phi(t)\equiv I_{2n}$, is also zero.
In the meanwhile,  our version of index is also different from Long index and Liu's $L_0$-index. For instance, consider the following degenerate path
\begin{align}\label{1.3}
	\Phi(t) = \begin{pmatrix}
		1 & 0 & 0 & 0\\
		0 & 1 & 0 & -t\\
		0 & 0 & 1 & 0\\
		0 & 0 & 0 & 1\\
	\end{pmatrix},\ 0 \leq t \leq 1.
\end{align}
The Long index (see Definition \ref{Def L}) and Liu's $L_0$-index (see Definition \ref{L0 index from I}) of this path are equal to $-1$, while by our Definition \ref{Def Maslov.type.index} its Maslov-type index  is $0$. We refer to the Example \ref{Eg degenerate paths} for more details.

Note that Robbin-Salamon also defined a version of  Maslov index for  general symplectic paths \cite{RS1}, while it is a half-integer. Instead, we intend to define an integer-valued index for the general symplectic paths by modifying methods of  \cite{SZ}. Since we release some conditions, the construction has to be ameliorated. Roughly, we consider perturbations to symplectic paths. Such idea appears in the works of Long\cite{L 1990, L0} and Cappell-Lee-Miller\cite{CLM}. Moreover, we will introduce other Maslov indices defined in \cite{CZ2}, \cite{SZ}, \cite{L0}, \cite{Liu1}, \cite{CLM}, \cite{RS1}, \cite{Le} and \cite{Go}  respectively, and compare them with ours. One of the main results is the following

\begin{thm}\label{Thm1}
	For any $\Phi \in \mathcal{P}(2n,\mathbb{R})$, there exists a Maslov-type index $\mu(\Phi)$ \eqref{4.14} %=\Delta(\Phi^{\#}_{\theta})+\Delta(\beta)$
	satisfying the following properties: \\
	(1) $\mu(\Phi)$ is an integer.  \\
	(2) $\mu(\Phi)$ is well defined, i.e. after the orthogonalization \eqref{4.3} of the two end points of $\Phi$, the index $\mu(\Phi)$ is independent of the choices of the global perturbations \eqref{4.8} and extensions \eqref{4.13}.\\
	(3) If $\Phi,\Psi$ are homotopic with fixed end points, then $\mu(\Phi)=\mu(\Psi)$. \\
	(4) $\forall\ 0<a<1$,  $\mu(\Phi)=\mu(\Phi([0,a]))+\mu(\Phi([a,1]))$.\\
	(5) If $(\mathbb{R}^{2n},\omega_0) = (V_1 \times V_2,\omega_1 \oplus \omega_2)$, then
	\begin{align*}
		\mu(\Phi) = \mu(\Phi_1) +\mu(\Phi_2)
	\end{align*}
	for any path $\Phi \in \mathcal{P}(2n,\mathbb{R})$ of the form $\Phi(z_1,z_2)=(\Phi_1z_1,\Phi_2z_2)$, where $\Phi_j$ is the path of $Sp(V_j,\omega_j),j=1,2$.\\
	(6) $\mu$ is of strong symplectic invariance, i.e 
	\begin{align*}
		\mu(T \Phi T^{-1}) = \mu(\Phi)
	\end{align*}
	for any paths $\Phi, T \in \mathcal{P}(2n,\mathbb{R})$.
\end{thm}

\begin{remark}\label{rk difference}
	The property (2) in the Theorem above does not mean that the index is independent of any perturbation. In fact, our construction need orthogonalization at two ends before perturbations.
	The reason we need orthogonalization (it does not change the rotation number which can refer to \eqref{orthogonal num hold}) is that after this manipulation the global perturbation can be chosen along a unique direction. If two different perturbations of end points are small enough, each end point can be deformed  into the same connected component \eqref{4.7}. Compared with the method of rotational perturbation, the degenerate end point might be deformed into  different connected components of $Sp_{1}^{*}(2n,\R)$ (see \eqref{3.1}) so that one may get some {\bf different values} of the index (see Example \ref{Eg L}). The Long index $\mu_{L}(\Phi)$ is then defined as the infimum of these different values. Set $\nu(\Phi):= \dim Ker(\Phi(1)-I)$,  the pair $(\mu_{L},\nu)$ is also called the Long index. $\nu$ gives the information of the end point and shows the variation range of index under the rotational perturbations. While for our construction, the global perturbation only engenders a {\bf unique value} of index and has no such variation range.
\end{remark}

Theorem \ref{Thm1} shows the main properties that our version of Maslov-type index satisfies. We note that, since the method of construction is different from the one by Long, this version of Maslov-type index in principle might not be determined by the axioms of Long index (Corollary 10 on Page 148 of \cite{L}), which are \textit{homotopy invariant, vanishing, symplectic additivity, catenation and normality}. It is pointed out in \cite{L} that if an index satisfies the first four axioms, then it is determined by the values in $Sp(2,\R)$ ($i.e.$ normality). Although our version of Maslov-type index satisfies some properties similar to Long index, including homotopy invariant (property (3)), catenation (property (4)) and symplectic additivity (property (5)),  it does not always satisfy vanishing and normality of Long's axioms. The vanishing axiom shows that Long index is equal to zero if $\nu(\Phi(t))= \dim Ker(\Phi(t)-I)$ is constant for any $0 \leq t \leq 1$. While our version of index may be not equal to zero because the so-called cycle we consider, denoted by
$$\mathscr{C}(2n,\mathbb{R}) := Sp_{1}(2n,\mathbb{R}) \cup Sp_{-1}(2n,\mathbb{R}),$$
which includes the components of symplectic matrices with eigenvalues $\pm{1}$(see \eqref{3.1}), is somehow different from the ordinary one. So such component may also contribute to the values of the index. If we want to obtain the similar ``vanishing" property, we should require that both $\nu(\Phi(t))= \dim Ker(\Phi(t)-I)$ and $\nu'(\Phi(t)):= \dim Ker(\Phi(t)+I)$ are constant, then the index is equal to zero. We do not formulate such a vanishing property since we at present are not sure whether there exists such a kind of axiom system for index derived from our method. Also do not we claim the normality, $i.e.$ the index in Theorem \ref{Thm1}  is determined by the values of paths in $Sp(2,\R)$.
Thus, even if the relationship between the index we defined and the Long index is clear for each symplectic path in $Sp(2,\R)$, we can not conclude the general relationship between two versions of  indices.
That is why we probably have to use example like (\ref{1.3}) to show the difference between different versions of indices.

\begin{remark}\label{rk cycle}
The cycle $\mathscr{C}(2n,\mathbb{R})$ that we considered as the degenerate set includes the components of eigenvalues $1$ and $-1$.
The motivation is that we intend to find an intuitive relation to the index of Lagrangian pairs and to provide a computational method that the index of Lagrangian pairs can correspond to the Maslov type index of the orthogonal symplectic paths. On the other hand, if we only consider the cycle as  the degenerate set of eigenvalue $1$, it is possible to establish a relation to the index of Lagrangian pairs in spaces of higher dimensions (see Remark \ref{rk cycle & Maslov cycle}). This relation can help to calculate the index of symplectic path via the one of Lagrangian pairs, but the reverse correspondence, $i.e.$  the general Lagrangian pairs' index via the one of symplectic path,  can not always be established. The more details will be explained in Section \ref{Sec-Mas-def}.
\end{remark}

\begin{remark}\label{rk omega index}
	In fact, we can also think of symplectic matrices with some prescribed eigenvalue lying on $\mathbb{S}^1\subset \C$ as the degenerate set  as Long's $\omega$-index, then our construction will also make sense. We may even consider the degenerate set that is determined case by case, depending on the starting point of symplectic path, while it will involve more complicated topological structure because the starting point may have some different eigenvalues. That might be a topic for subsequent research.
\end{remark}

Then we compare the index we defined with other Maslov-type indices, $i.e.$  Conley-Zehnder index $\mu_{CZ}$ (Definition \ref{Def CZ}), Long index $\mu_{L}$ (Definition \ref{Def L}) , Long's SPS index $\hat{\mu}_{L}$\eqref{Long SPS}, Liu's $L_0$-index $i_{L_0}$ (Definition \ref{L0 index from I}) , Liu's SPS index $\hat{i}_{L_0}$ \eqref{Liu SPS}, Cappell-Lee-Miller index $\mu_{CLM}$ (Definition \ref{Def-CLM} and \ref{Def-CLM-sp}), Robbin-Salamon index $\mu_{RS}$ \eqref{3.17} and Leray-Gosson index $\mu_{LG}$ (Definition \ref{Def LG}).

\begin{thm}[Comparison with Conley-Zehnder-Long index]\label{Thm2}
	If $\Phi \in \mathcal{P}(2n,\mathbb{R})$ satisfies $\Phi(0)=I$, then
	\begin{align}\label{1.4}
		\mu_{CZL}(\Phi)=\mu(\Phi)-r(\Phi(1))-l(\Phi(1)).
	\end{align}
	where for the matrix $\Phi(1)$, $r(\cdot)$ counts the number of the first kind eigenvalues (see Definition \ref{Def first kind}) on $\mathbb{S}^{1}$  with negative imaginary part (see \eqref{5.7}) and $l(\cdot)$ is an integer caused by Long's operation of rotation perturbation for  $\Phi(1)$ (see \eqref{l}). In particular, if
    $\det (I-\Phi(1)) \not = 0$, then $l(\Phi(1))=0$ and we have
	\begin{align}\label{1.5}
		\mu_{CZ}(\Phi)=\mu(\Phi)-r(\Phi(1)).
	\end{align}
\end{thm}
\begin{remark}\label{Thm2 rk}
	The function $r$ implies the non-degenerate information at the endpoint, while $l$ implies degenerate information, which naturally disappears if the definition is restricted to the non-degenerate paths.
\end{remark}
\begin{thm}[Comparison with $L_0$-index]\label{Thm3}
	If $\Phi \in \mathcal{P}(2n,\mathbb{R})$ satisfies $\Phi(0)=I$, then
	\begin{align}\label{1.6}
		i_{L_0}(\Phi)=\mu(\Phi)-r(\Phi(1))-l(\Phi(1))-c(\Phi(1)),
	\end{align}
	where $c(M)$ is the $L_0$-concavity (see \eqref{3.10}) of a symplectic matrix $M$.
\end{thm}
\begin{thm}[Comparison with SPS index]\label{Thm4}
	If $\Phi \in \mathcal{P}(2n,\mathbb{R})$ is a general path (it also can be viewed as a segment of symplectic paths for the SPS index), then
	\begin{align}\label{1.7}
		&\hat{\mu}_{L}(\Phi)=\mu(\Phi)+r(\Phi(0))-r(\Phi(1))+l(\Phi(0))-l(\Phi(1)), \\
		&\hat{i}_{L_0}(\Phi)=\mu(\Phi)+r(\Phi(0))-r(\Phi(1))+l(\Phi(0))-l(\Phi(1))+c(\Phi(0))-c(\Phi(1)).
	\end{align}
\end{thm}

Then for a symplectic vector space $(\R^{2n},\omega_0)$, we consider paths of Lagrangian subspaces $L(t)$, $t\in [0,1]$.
As we mentioned, the Cappell-Lee-Miller index for Lagrangian path pairs $f(t):=(L_1(t), L_2(t))$, denoted by $\mu_{CLM}(f)$,
can naturally induce an index for general symplectic paths, denoted by $\mu_{CLM}(\Phi)$(see Definitions \ref{Def-CLM} and \ref{Def-CLM-sp}).
Then we have
\begin{thm}[Comparison with Cappell-Lee-Miller index]\label{Thm5}
   Let $f(t)=(L_1, L_2(t))=(L_1, \Phi(t) L_1)$ be a Lagrangian pair, where $L_1=\mathbb{R}^{n} \times \{0\}$ and $\Phi$ is a symplectic path. Then there exists a corresponding orthogonal symplectic path $O$  such that
	\begin{align}\label{1.8}
		\mu_{CLM}(\Phi):=\mu_{CLM}(f)=\mu(O).
	\end{align}
\end{thm}

\begin{remark}\label{rk relation to CLM}
From the result above, the Cappell-Lee-Miller index can be calculated by our Maslov-type index of some special symplectic path. Recently, the relationship between $\mu_{CLM}$ and $\mu$ has been further studied by Cappell and Miller \cite{CM}. They express these two indices in terms of integral plus end point corrections and give a nice formula for the relation between them (see Theorem 1.5 and section 6 of \cite{CM}).
\end{remark}

\begin{thm}[Comparison with Robbin-Salamon index, Leray-Gosson index]\label{Thm6}
 For a symplectic path $\Phi$, there exists a corresponding orthogonal symplectic path $O$  such that
	\begin{align}\label{1.9}
		\mu_{LG}(\Phi)=\mu_{RS}(\Phi)=\mu(O)+\frac{1}{2}(s(0)-s(1)),
	\end{align}
	where $\mu_{LG}(\Phi)$ (see \eqref{LG symplectic} and \eqref{LG=RS symplectic}) and $\mu_{RS}(\Phi)$ (see \eqref{3.16}, \eqref{3.17}) are the Leray-Gosson index and Robbin-Salamon index for symplectic path $\Phi$ respectively. $s(0)$ and  $s(1)$ are the ranks of  crossing forms for $O(0)$and $O(1)$ respectively (see and \eqref{5.8}).
\end{thm}

\bigskip

In section \ref{Sec-Prel}, we first recall some facts about symplectic matrices and introduce some tools, which are used to construct the index. 
In section \ref{Sec-rev}, we review several other  Maslov-type indices and show the ideas and methods of their definitions.
Then we show our definition in section \ref{Sec-Mas-def} via the tools from section \ref{Sec-Prel} and some ideas from section \ref{Sec-rev}. 
Finally, in section \ref{Sec-proof}, we prove the main results  and give two concrete examples  to show the interrelationships of different indices.

\bigskip

\noindent {\bf Acknowledgements.} The authors would like to thank Yiming Long and Chungen Liu for pointing out some important references, 
helpful communications and invaluable suggestions for improving the manuscript. They also want to thank Duanzhi Zhang for inspiring conversations. 
This paper was posted on arXiv  in July 2024. Then the authors received the reflection from Sylvain E. Cappell and Edward Y. Miller in July 2026. 
The authors are grateful to Cappell and Miller for their positive comments on this work, kindly providing their new manuscript related to this paper and helpful suggestion for improving this work.
Both authors are partially supported by  the project No. 11671209 of National Science Foundation of China 
and the project No. 2021A1515010379 of Guangdong Basic and Applied Basic Research Foundation, China.

\section{Preliminaries}\label{Sec-Prel}	
In this section, we introduce some definitions and results that we use in the article.

\begin{definition}\label{Def symp.sp.}
	Let $V$ be a vector space of $2n$ dimension and $\omega$  a bilinear form in $V$ satisfying:
	\begin{align*}
		&(1)\ \forall \ \xi,\eta \in V,\ \omega(\xi,\eta) = -\omega(\eta,\xi),\\
		&(2)\ \text{If} \ \forall \ \xi \in V,\ \omega(\xi,\eta) = 0,\ \text{then} \ \eta=0.
	\end{align*}
	Then the space $(V,\omega)$ is called a symplectic space. This bilinear form $\omega$ is called the symplectic form of $(V,\omega)$.
\end{definition}

\begin{definition}\label{Def symp.iso.}
	Let $T: (V_1,\omega_1) \to (V_2,\omega_2)$ be a linear map. $T$ is called a symplectic isomorphism if T is an isomorphism and has the pull-back $T^*\omega_2 = \omega_1$.
\end{definition}
The set of all symplectic isomorphisms of $(V,\omega)$ with composition can be regarded as a group, called symplectic group \cite{L,MD}, and is denoted by $Sp(V,\omega)$.
A  continuous map
\begin{align*}
	\Phi:[0,1] \to Sp(V,\omega)
\end{align*}
is called a symplectic path in $Sp(V,\omega)$. If $\Phi(0)=\Phi(1)$, then $\Phi$ is called a symplectic loop. If two symplectic paths $\Phi, \Psi$ satisfies $\Phi(1) = \Psi(0)$, then they have the catenation defined by
\begin{equation}\label{catenation}
	\Phi \# \Psi(t) := \left\{
	\begin{array}{ll}
		\Phi(2t) & 0 \leq t < \frac{1}{2} \\
		\Psi(2t-1) & \frac{1}{2} \leq t \leq 1
	\end{array} \right.	.
\end{equation}

\begin{definition}\label{Lagrangian subsp}
	Let $L$ be a $n$-dimensional subspace of the symplectic space $(V,\omega)$ and
	\begin{align*}
		L^{\bot} = \{v \in V |\ \omega(v,w)=0,\ \forall w \in L\},
	\end{align*}
	where $L^{\bot}$ is called the skew-orthogonal complement \cite{A2} of $L$. If $L=L^{\bot}$, then $L$ is called a Lagrangian subspace of $(V,\omega)$.
\end{definition}
The set of all Lagrangian subspaces of $V$ is called Lagrangian Grassmannian of $V$, denoted by $\mathscr{L}(V)$. Then $\mathscr{L}(V)$ is a homogeneous space of $Sp(V,\omega)$. A continuous map $L:[0,1] \to \mathscr{L}(V)$ is called a Lagrangian path and is called Lagrangian loop if $L(0)=L(1)$.

Now we consider the symplectic space $(\mathbb{R}^{2n},\omega_0)$
where\begin{align*}
	\omega_0= {\sum\limits_{j=1}^{n}{dy}^{j}\wedge{dx}^{j}},
\end{align*} is the standard symplectic form, where $(x^1,\ldots,x^n,y^1,\ldots,y^n)$ are the coordinates of $(\mathbb{R}^{2n},\omega_0)$.
Denote the symplectic group of $(\mathbb{R}^{2n},\omega_0)$ by $Sp(2n,\mathbb{R})$, $M \in Sp(2n,\mathbb{R})$ is a $2n \times 2n$ real matrix and it satisfies $M^{T}J_{0}M = J_{0}$, where
\begin{align*}
	J_{0} =
	\begin{pmatrix}
		O & I_n\\
		-I_n & O\\
	\end{pmatrix}.
\end{align*}
Then we have the following definition:

\begin{definition}\label{Def symp.matrix}
	Let $M \in \mathbb{R}^{2n \times 2n}$. $M$ is called a symplectic matrix if it satisfies
	\begin{align*}
		M^{T}J_{0}M = J_{0}.
	\end{align*}
\end{definition}

We continue to introduce some properties about symplectic matrices and the following proposition holds  (also see \cite{L,MD}).

\begin{proposition}\label{Prop symp.matrix}
	For a symplectic matrix $M$, denote the set of all eigenvalues of $M$ by $\sigma(M)$. Then we have
	(1) $\det M = 1$.
	(2) If $\lambda \in \sigma(M)$, then $\lambda^{-1} \in \sigma(M)$, i.e. $M$ has the pairs $\{\lambda,\lambda^{-1}\}$ of eigenvalues. If $\lambda \in \sigma(M) \cap \mathbb{S}^{1}$, then $M$ has the pairs $\{\lambda,\bar{\lambda}\}$ of eigenvalues.
\end{proposition}

Now we introduce the first kind eigenvalue \cite{SZ} of a symplectic matrix $M$. We view $M$ as a map from $\mathbb{C}^{2n}$ to $\mathbb{C}^{2n}$, let $\lambda \in \sigma(M)$  be an eigenvalue of multiplicity $m(\lambda)$, the generalized eigenspace
$$E_{\lambda}(M) = \bigcup\limits_{j=1}^{m(\lambda)} Ker(\lambda I - M)^{j}$$
is a subspace of $\mathbb{C}^{2n}$. The action of $\omega_0$ on $E_{\lambda}(M) \times E_{\lambda}(M)$ is given by
\begin{align*}
	\omega_0(\xi_1,\xi_2)=(J_{0} \xi_1)^{T}\xi_2,\ \forall \xi_1,\xi_2 \in E_{\lambda}(M).
\end{align*}
For $\forall \lambda \in \sigma(M) \cap \mathbb{S}^{1} \backslash \{\pm 1\}$, define a bilinear form
\begin{align*}
	Q_{\lambda}(\xi_1,\xi_2) = Im \, \omega_0(\bar{\xi_1},\xi_2)
\end{align*}
on $E_{\lambda}(M)$. Since
\begin{align*}
	Q_{\lambda}(\xi_1,\xi_2)-Q_{\lambda}(\xi_2,\xi_1)
	&= Im[\omega_0(\bar{\xi_1},\xi_2)-\omega_0(\bar{\xi_2},\xi_1)]\\
	&=-Im[\overline{\omega_0(\bar{\xi_1},\xi_2)}-\overline{\omega_0(\bar{\xi_2},\xi_1)}]\\
	&=-Im[\omega_0(\xi_1,\bar{\xi_2})-\omega_0(\xi_2,\bar{\xi_1})]\\
	&=-Im[\omega_0(\bar{\xi_1},\xi_2)-\omega_0(\bar{\xi_2},\xi_1)],
\end{align*}
then $Q_{\lambda}(\xi_1,\xi_2)-Q_{\lambda}(\xi_2,\xi_1) = 0$, so $Q_{\lambda}$ is a non-degenerate symmetric bilinear form, hence $Q_{\lambda}$ divides $E_{\lambda}(M)$ into two subspaces $E_{\lambda}^{+}(M)$ and $E_{\lambda}^{-}(M)$ such that
\begin{align*}
	Q_{\lambda}(\xi,\xi) > 0,\ \forall \xi \in E_{\lambda}^{+}(M) \backslash \{0\}; \\
	Q_{\lambda}(\xi,\xi) < 0,\ \forall \xi \in E_{\lambda}^{-}(M) \backslash \{0\}.
\end{align*}
For an eigenvalue pair $\{\lambda, \bar{\lambda}\}$, the eigenvector $\xi \in E_{\lambda}^{+}(M)$ if and only if $\bar{\xi} \in E_{\lambda}^{-}(M)$ because of
\begin{align*}
	Q_{\lambda}(\bar\xi,\bar\xi)=Im \, \omega_0(\xi,\bar{\xi})=-Im \, \omega_0(\bar{\xi},\xi)=-Q_{\lambda}(\xi,\xi),
\end{align*}
then we have $E_{\lambda}^{+}(M)=E_{\bar{\lambda}}^{-}(M)$. Since the identity $Q_{\lambda}(i\xi_1,i\xi_2) = Q_{\lambda}(\xi_1,\xi_2)$, both
$E_{\lambda}^{+}(M)$ and $E_{\lambda}^{-}(M)$ are of the even dimension. Set $\dim E_{\lambda}^{+}(M) = 2m^{+}(\lambda)$. 
Then following Gelfand and Lidskii\cite{GL} and Salamon and Zehnder \cite{SZ}, 
one has the definition of eigenvalues of the first kind:
\begin{definition}\label{Def first kind}
	$\lambda \in \sigma(M)$ is called an {\bf eigenvalue of the first kind} of $M$ if it satisfies one of the following conditions:
	\begin{align*}
		&(1) \ \lambda = \pm{1} \ \text{or} \ |\lambda| < 1;
		&(2)\  \lambda \in \mathbb{S}^{1} \backslash \{\pm{1}\} \ \text{and} \ m^{+}(\lambda)>0.
	\end{align*}
\end{definition}

\begin{remark}
In contrast to the definition given in \cite{SZ}, Definition \ref{Def first kind} takes $\pm{1}$ into account for later elaboration, but this does not affect the construction of our subsequent definitions.  
\end{remark}
If all eigenvalues of $M$ are distinguishable, then we can order all eigenvalues as
\begin{align*}
	\lambda_1,\ \lambda_2,\ ...,\  \lambda_n,\ \lambda^{-1}_1,\ \lambda^{-1}_2,\ ...,\  \lambda^{-1}_n,
\end{align*}
where $\lambda_1, \lambda_2, ...,\lambda_n$ are eigenvalues of the first kind, then we can define a map $\rho:Sp(2n,\mathbb{R}) \to \mathbb{S}^{1}$ as
\begin{align}\label{2.1}
	\rho(M) = \prod\limits_{j=1}^{n} \frac{\lambda_j}{|\lambda_j|}.
\end{align}
For general case,  denote by $m^{+}(\lambda)$ the multiplicity of the eigenvalue $\lambda$ of the first kind, denote the number of pairs \{$\lambda$, $\lambda^{-1}$\} of negative eigenvalues by $m_0$. 
According to the Theorem 3.1 of \cite{SZ}, we have

\begin{theorem}[Salamon-Zehnder]\label{Thm SZ}
	There is a unique continuous mapping
	\begin{align*}
		\rho : Sp(2n,\R) \to \mathbb{S}^{1},
	\end{align*}
	given by
	\begin{align}\label{2.2}
		\rho(M) = (-1)^{m_0}\prod\limits_{\lambda \in \sigma(M) \cap \mathbb{S}^{1} \backslash \{\pm{1}\}}\lambda^{m^{+}(\lambda)},
	\end{align}
	satisfying the following properties: \\
	(1) Naturality: If $T:Sp(2n,\mathbb{R}) \to Sp(V,\omega)$ is a symplectic isomorphism, then
	\begin{align}\label{2.3}
		\rho(TMT^{-1}) = \rho(M)
	\end{align}
	for any $M \in Sp(2n,\mathbb{R})$. \\
	(2) Product: If $(\mathbb{R}^{2n},\omega_0) = (V_1 \times V_2,\omega_1 \oplus \omega_2)$, then
	\begin{align}\label{2.4}
		\rho(M) = \rho(M^{'})\rho(M^{''})
	\end{align}
	for any $M \in Sp(2n,\mathbb{R})$ of the form
	$M(z_1,z_2)=(M^{'}z_1,M^{''}z_2)$, where $M^{'} \in Sp(V_1,\omega_1)$ and $M^{''} \in Sp(V_2,\omega_2)$. \\	
	(3) Determinant: If $M \in Sp(2n,\mathbb{R}) \cap O(2n)$(i.e. the orthogonal group) is of the form
	\begin{align*}
		\begin{pmatrix}
			X & -Y \\
			Y & X
		\end{pmatrix} ,
	\end{align*}
	where $X^{T}Y = Y^{T}X$ and $X^{T}X+Y^{T}Y = I$, then
	\begin{align}\label{2.5}
		\rho(M) = \det(X+iY).
	\end{align} 	
	(4) Normalization:  If $M$ has no eigenvalue on $\mathbb{S}^{1}$, then
	\begin{align*}
		\rho(M) = \pm{1}.
	\end{align*}
\end{theorem}

For any symplectic path $\Phi :[0,1] \to Sp(2n,\mathbb{R})$, the map $\rho(\Phi)$ is continuous, then there exists a continuous map $\alpha :[0,1] \to \mathbb{R}$ such that
\begin{align}\label{2.6}
	\rho(\Phi(t))=e^{i\pi \alpha(t)}.
\end{align}
Define the rotation number of the path $\Phi$ from time $0$ to $t$ as
\begin{align}\label{2.7}
	\Delta(\Phi(t)) = \alpha(t) - \alpha(0)
\end{align}
and simply write $\Delta(\Phi)=\Delta(\Phi(1))$. $\Delta$ is the important tools to define index and it has some properties about the homotopy of paths. Let $\Phi, \Psi$ be two paths in $Sp(2n,\mathbb{R})$, we call $\Phi$ and $\Psi$ are homotopic if there exists a continuous map $H(t, s)$ on $[0,1] \times [0,1]$ such that
\begin{align*}
	H(t,0)=\Phi(t), \ H(t,1)=\Psi(t).
\end{align*}
If a loop in $Sp(2n,\mathbb{R})$ is homotopic to a point, then we say this loop is contractible.

\begin{proposition}\label{Prop Delta}
	The rotation number $\Delta$ has the following properties:\\
	(1) If $\Phi$ is a symplectic loop, then $\Delta(\Phi) \in \mathbb{Z}$. In particular, if $\Phi$ is contractible, then
	\begin{align}\label{2.8}
		\Delta(\Phi)=0.
	\end{align}
	(2) If $0<a<1$, then
	\begin{align}\label{2.9}
		\Delta(\Phi) = \Delta(\Phi([0,a])) + \Delta(\Phi([a,1])).
	\end{align}
	(3) If $\Phi ,\Psi$ are two homotopic symplectic paths with fixed endpoints, then
	\begin{align}\label{2.10}
		\Delta(\Phi)=\Delta(\Psi).
	\end{align}
	(4) If $T:Sp(2n,\mathbb{R}) \to Sp(V,\omega)$ is a symplectic isomorphism, then
	\begin{align}\label{2.11}
		\Delta(T\Phi T^{-1}) =\Delta(\Phi).
	\end{align}
	(5 )If $(\mathbb{R}^{2n},\omega_0) = (V_1 \times V_2,\omega_1 \oplus \omega_2)$, then
	\begin{align}\label{2.12}
		\Delta(\Phi) = \Delta(\Phi_1) +\Delta(\Phi_2)
	\end{align}
	for any path $\Phi \in Sp(2n,\mathbb{R})$ of the form $\Phi(z_1,z_2)=(\Phi_1z_1,\Phi_2z_2)$, where $\Phi_j$ is the path of $Sp(V_j,\omega_j),\ j=1,2$.
\end{proposition}

\begin{proof}
	(1) Since $\rho(\Phi(0))=\rho(\Phi(1))$, then $e^{i\pi(\alpha(1) - \alpha(0))}=1$ and we have
	\begin{align*}
		\Delta(\Phi) = \alpha(1) - \alpha(0) \in \mathbb{Z}.
	\end{align*}
	If $\Phi$ is contractible, then $\rho(\Phi(t))$ is contractible on $\mathbb{S}^{1}$, then $\Delta(\Phi)=0$. \\
	(2)	$\Delta(\Phi) = \alpha(1)-\alpha(0) = (\alpha(1)-\alpha(a))+(\alpha(a)-\alpha(0))=\Delta(\Phi |_{[0,a]}) + \Delta(\Phi |_{[a,1]})$.\\
	(3) Since $\Phi$ and $\Psi$ have the same end points, then
	\begin{equation*}
		\Phi \# (-\Psi(t)) := \left\{
		\begin{array}{ll}
			\Phi(2t) & 0 \leq t < \frac{1}{2} \\
			\Psi(2-2t) & \frac{1}{2} \leq t \leq 1
		\end{array} \right.	
	\end{equation*}
	is a contractible loop, where $-\Psi(t) := \Psi(1-t)$ is the reverse path. It follows from \eqref{2.8} and \eqref{2.9} that
	\begin{align*}
		\Delta(\Phi)-\Delta(\Psi)=\Delta(\Phi \# -\Psi) = 0
	\end{align*}
	and then $\Delta(\Phi)=\Delta(\Psi)$.\\
	(4) According to \eqref{2.3} and \eqref{2.7}, this property has been proved.\\
	(5) By \eqref{2.4}, we have $\rho(\Phi) = \rho(\Phi_1)\rho(\Phi_2)$, then there exists $\alpha_j :[0,1] \to \mathbb{R},\ j=1,2$ such that
	$$\rho(\Phi_j(t))=e^{i\pi \alpha_j(t)}, \ j=1,2.$$
	Then
	$$\rho(\Phi(t))=e^{i\pi (\alpha_1(t)+\alpha_2(t))},$$
	and hence
	$$\Delta(\Phi) = (\alpha_1(1)+\alpha_2(1))-(\alpha_1(0)+\alpha_2(0)) = \Delta(\Phi_1)+\Delta(\Phi_2).$$
	This completes the proof.
\end{proof}

\begin{remark}\label{CZ rotation num}
	The rotation number above is defined by Salamon and Zehnder \cite{SZ}. There is another version defined by Conley and Zehnder \cite{CZ2} and it also has those same properties as ones in Proposition \ref{Prop Delta}.
Recall for any symplectic path $\Phi=\Phi(t)$, it can be represented in polar form as
	\begin{equation}\label{polar form}
		\Phi=PO,
	\end{equation}
	where $P=(\Phi \Phi^T)^{1/2}$ is a positive definite symmetric and symplectic path and $O=P^{-1}\Phi$ is an orthogonal symplectic path which has the form as
	\begin{align*}
		\begin{pmatrix}
			X & -Y \\
			Y & X
		\end{pmatrix} ,
	\end{align*}
	where $X^{T}Y = Y^{T}X$ and $X^{T}X+Y^{T}Y = I$. Using this unique form of $O$, one can directly define a number for each $t$ as
$$\rho'(\Phi(t)):=\det(X(t)+iY(t))$$
and choose a continuous map $\alpha' :[0,1] \to \mathbb{R}$ such that $\rho'(\Phi(t))=e^{i\pi \alpha'(t)}$.
Then the rotation number of Conley-Zehnder version is defined by
$$\Delta'(\Phi): = \alpha'(1) - \alpha'(0)\ \ (=\Delta(O)).$$
Here we consider two special cases of $M \in Sp(2n,\R)$.

{\rm (1)} If $M$ is orthogonal and symplectic, then the positive definite symmetric and symplectic matrix  $P_M$ in the polar form is just identity $I_{2n}$.

{\rm (2)} On the other hand, if $$M = {\rm diag}\{\lambda_1,\cdots,\lambda_n,\lambda_{1}^{-1},\cdots,\lambda_{n}^{-1} \}$$ is diagonal and symplectic, then
$$P_M = {\rm diag}\{ |\lambda_1|,\cdots,|\lambda_n|,|\lambda^{-1}_{1}|,\cdots,|\lambda^{-1}_{n}| \}.$$
Since the set of positive definite symmetric and  symplectic matrices is contractible, if $M=P_{M}O_{M}$ is of polar form and $M$ is a diagonal symplectic matrix (or an orthogonal symplectic matrix ),
one can easily choose a positive definite symmetric, diagonal and symplectic path $P$ such that
	$P(0)=P_M,P(1)=I_{2n}$. Then  $\Psi(t):=P(t)O_M$ is a symplectic path starting from $M$ and ending at $O_M$.
We can see that all eigenvalues of $\Psi(t)=P(t)O_M$ are of the form as $l(t)e^{i\theta}$ $(l(t)>0)$.
Then by \eqref{2.2},\eqref{2.6} and \eqref{2.7}, we obtain that $\Delta(\Psi)=0$.

	If the two endpoints $\Phi(0)$ and $\Phi(1)$ are diagonal symplectic matrices (or orthogonal symplectic matrices ), then there exists $\beta_1$ and $\beta_2$ as $\Psi$ above such that $-\beta_1 \# \Phi \# \beta_2$ is homotopy to $O$ with fixed endpoints. By Proposition \ref{Prop Delta} (2), (3) and $\Delta(\beta_1)=\Delta(\beta_2)=0$, we have
$$\Delta'(\Phi)=\Delta(O)=\Delta(-\beta_1 \# \Phi \# \beta_2)=\Delta(\Phi).$$
This means that these two rotation numbers are equivalent to each other at least for those symplectic paths with diagonal or orthogonal endpoints.

\end{remark}

\section{Review of various  indices}\label{Sec-rev}

In this section we will introduce six versions of Maslov-type indices and show their main ideas of construction.
For $\lambda \in \mathbb{S}^1 \subset \C$, we use the following notations, which are first introduced by Long \cite{L' 1999}:
\begin{align}\label{3.1}	
	&Sp_{\lambda}(2n,\mathbb{R}):=\{M \in Sp(2n,\mathbb{R})| \bar{\lambda}^{n}\det(\lambda I-M) =0\},\\
	&Sp^{*}_{\lambda}(2n,\mathbb{R}):=\{M \in Sp(2n,\mathbb{R})| \bar{\lambda}^{n}\det(\lambda I-M) \not =0\},\\
	&Sp_{\lambda}^{\pm}(2n,\mathbb{R}):=\{M \in Sp(2n,\mathbb{R}) |\pm{(-1)^{n-1}}\bar{\lambda}^{n}\det(\lambda I-M) < 0\}.
\end{align}

\subsection{Conley-Zehnder-Long index}\label{rev CZL index}
The first kind of indices was studied by Conley and Zehnder\cite{CZ2,SZ} and Long\cite{L 1990,L0}, which originates from the study of periodic solutions of Hamiltonian Equations. Such a solution generates a symplectic path as
\begin{align}\label{3.2}
	\Phi :[0,1] \to Sp(2n,\mathbb{R}),\ \Phi(0)=I.
\end{align}
The path as \eqref{3.2} is called the non-degenerate path if it satisfies $\det (I-\Phi(1)) \not=0$ and called the degenerate path if it satisfies $\det (I-\Phi(1)) =0$.
By Theorem \ref{Thm SZ}, the symplectic path $\Phi$ corresponds to a path $\rho(\Phi)$ on $\mathbb{S}^{1}$ and obtains a number $\Delta(\Phi)$ (\ref{2.7}) but not always an integer.
If $\Phi$ is a non-degenerate path, the idea of constructing the Conley-Zehnder index is that we need to
find a suitable extension $\gamma$ for $\Phi$ such that $\Delta(\Phi) + \Delta(\gamma)$ is an integer.

Since $W^{+}=-I$ and $\displaystyle W^{-}={\rm diag}\{2,-1,\cdots,-1,\frac12,-1,\cdots,-1\}$ are in the different connected components of $Sp_{1}^{*}(2n,\mathbb{R})$,
then we define the extension
\begin{align}\label{ext-gamma}
	\gamma :[0,1] \to Sp_{1}^{*}(2n,\mathbb{R}),\ \gamma(0)=\Phi(1),\ \gamma(1) \in \{W^{+}, W^{-}\}.
\end{align}
Thus we have the following definition \cite{CZ2}:

\begin{definition}[Conley-Zehnder index]\label{Def CZ}
	For any non-degenerate path $\Phi$, the Conley-Zehnder index for $\Phi$ is defined by
	\begin{align}\label{3.3}
		\mu_{CZ}(\Phi) = \Delta(\Phi) + \Delta(\gamma).
	\end{align}
\end{definition}
By Remark \ref{CZ rotation num}, the rotation number of different versions are equivalent for Conley-Zehnder index and we use the version defined by Salamon and Zehnder \cite{SZ}. According to the Lemma 3.2 of \cite{SZ} , every loop in $Sp^{*}_{1}(2n,\mathbb{R})$ is contractible. If we choose another extension $\gamma'$, by Proposition \ref{Prop Delta} (1), $\Delta(\gamma' \# -\gamma)=0$ and hence $\Delta(\gamma)=\Delta(\gamma')$. We see that $\Delta(\gamma)$ depends only on the terminal point $\Phi(1)$.
Then the index $\mu_{CZ}(\Phi)$ is independent of the choices of $\gamma$ so that it is well defined.
\begin{example} \label{Eg1}
	Let
	\begin{align*}
		\Phi(t) = \begin{pmatrix}
			\cos\frac{3\pi t}{2} & -\sin\frac{3\pi t}{2} \\
			\sin\frac{3\pi t}{2} & \cos\frac{3\pi t}{2} \\
		\end{pmatrix},\ 0 \leq t \leq 1.
	\end{align*}
	We can see $\Phi(1) \in Sp_{1}^{+}(2,\mathbb{R})$, choose the extension as
	\begin{align*}
		\gamma(t) =
		\begin{pmatrix}
			\cos\frac{3-t}{2}\pi & -\sin\frac{3-t}{2}\pi \\
			\sin\frac{3-t}{2}\pi & \cos\frac{3-t}{2}\pi \\	
		\end{pmatrix}		
	\end{align*}				
	By Theorem \ref{Thm SZ} (3), we have
	\begin{align*}
		&\rho(\Phi)=\det(\cos\frac{3\pi t}{2}+i\sin\frac{3\pi t}{2})=e^{\frac{3\pi ti}{2}}, \\
		&\rho(\gamma)=\det(\cos\frac{(3-t)\pi}{2}+i\sin\frac{(3-t)\pi}{2})= e^{\frac{(3-t)\pi i}{2}},
	\end{align*}
	then the Conley-Zehnder index
	\begin{align*}
		\mu_{CZ}(\Phi)=\Delta(\Phi)+\Delta(\gamma)=\frac{3}{2}+(-\frac{1}{2})=1.
	\end{align*}
\end{example}
\begin{remark}\label{remark}
	The definitions of Conley-Zehnder index can not apply directly  for those symplectic paths which do not start at identity. For example, let
	\begin{align*}
		\Phi(t) = \begin{pmatrix}
			\cos \pi (t+\frac{1}{2}) & -\sin \pi (t+\frac{1}{2}) \\
			\sin \pi (t+\frac{1}{2}) & \cos \pi (t+\frac{1}{2}) \\
		\end{pmatrix},\ 0 \leq t \leq 1.
	\end{align*}
	Since $\Phi(1) \in Sp^{+}_{1}(2,\mathbb{R})$, the end point of the extension is ${\rm diag}\{-2,-\frac{1}{2}\}$.
	If we construct this extension,  the``generalized Conley-Zehnder index" should be equal to $\frac{1}{2}$,
	which is not an integer. But in our definition (see \eqref{4.14}), it has the index $1$.	
\end{remark}

For degenerate paths (i.e. $\det(I-\Phi(1))=0$), Long\cite{L 1990,L0} used the method of rotational perturbation to deal with this case. To show more about this method, we introduce some notations and results from \cite{L0,L}. Firstly, for any two real matrices of the square block form
\begin{align*}
	M_1=\begin{pmatrix}
		A_1 & B_1\\
		C_1 & D_1
	\end{pmatrix}_{2j \times 2j},\ M_2=\begin{pmatrix}
		A_2 & B_2\\
		C_2 & D_2
	\end{pmatrix}_{2k \times 2k},
\end{align*}
we define their $ \diamond$-product by
\begin{align*}
	M_1 \diamond M_2=\begin{pmatrix}
		A_1 & 0 & B_1 & 0\\
		0 & A_2 & 0 & B_2\\
		C_1 & 0 & D_1 & 0\\
		0 & C_2 & 0 & D_2
	\end{pmatrix}_{2(j+k) \times 2(j+k)}
\end{align*} 
and denote by $M^{\diamond j}$ the $j$-fold $\diamond$-product $M \diamond \dots \diamond M$.
For $k=1$, we define the normal form of eigenvalue 1 as
	\begin{align}\label{N1}
		N_k(\bb)=N_1(b)=\begin{pmatrix}
			1 & b\\
			0 & 1
		\end{pmatrix},\  \bb=b=0, \ \pm{1}.
	\end{align}
For $ k \geq 2 $, the normal form is defined as
	\begin{align}\label{Nk}
		N_k(\bb)=\begin{pmatrix}
			A_k(1) & B_k(\bb)\\
			0 & C_k(1)
		\end{pmatrix},
	\end{align}
	where $ A_k(1) $ is a $ k\times k $ Jordan block of the eigenvalue $1$:
	\begin{align}\label{Ak}
		A_k(1)=\begin{pmatrix}
			1 & 1 & 0 & \cdots & 0 & 0\\
			0 & 1 & 1 & \cdots & 0 & 0\\
			\cdot & \cdot & \cdot & \cdots & \cdot & \cdot\\
			0 & 0 & 0 & \cdots & 1 & 1\\
			0 & 0 & 0 & \cdots & 0 & 1
		\end{pmatrix},
	\end{align}
	$ C_k(1) $ is a $ k\times k $ lower triangle matrix of the following form:
	\begin{align}\label{Ck}
		C_k(1)=\begin{pmatrix}
			1 & 0 & 0 & \cdots & 0 & 0\\
			-1 & 1 & 0 & \cdots & 0 & 0\\
			1 & -1 & 1 & \cdots & 0 & 0\\
			\cdot & \cdot & \cdot & \cdots & \cdot & \cdot\\
			\cdot & \cdot & \cdot & \cdots & \cdot & \cdot\\
			(-1)^{k-1} & (-1)^{k-2} & (-1)^{k-3} & \cdots & 1 & 0\\
			(-1)^{k} & (-1)^{k-1} & (-1)^{k-2} & \cdots & -1 & 1\\
		\end{pmatrix},
	\end{align}
	$ B_k(\bb) $ is a $ k\times k $ lower triangle matrix of the following form with $ \bb=(b_1,\cdots, b_k)\in\mathbb{R}^k $:
	\begin{align}\label{Bk}
		B_k(\bb)=\begin{pmatrix}
			b_1 & 0 & 0 & \cdots & 0 & 0\\
			b_2 & -b_2 & 0 & \cdots & 0 & 0\\
			b_3 & -b_3 & b_3 & \cdots & 0 & 0\\
			\cdot & \cdot & \cdot & \cdots & \cdot & \cdot\\
			\cdot & \cdot & \cdot & \cdots & \cdot & \cdot\\
			b_{k-1} & -b_{k-1} & b_{k-1} & \cdots &  (-1)^{k-2}b_{k-1} & 0\\
			b_{k} & -b_{k} & b_{k} & \cdots &  (-1)^{k-2}b_{k} & (-1)^{k-1}b_{k}\\
		\end{pmatrix}.
	\end{align}
For more details about the normal forms,  see Section 7 of \cite{L0} or Section 1.4 in Long's book \cite{L}. 
According to Theorem 7.3 of \cite{L0}, one has the following
\begin{proposition}\label{standard form}
	For any $M \in Sp_{1}(2n,\R)$, there exists $P \in Sp(2n,\R)$ such that
	\begin{align*}
			PMP^{-1} = N_{k_1}(\bb_1) \diamond N_{k_2}(\bb_2) \diamond \cdots \diamond N_{k_q}(\bb_q) \diamond M_0,
	\end{align*} 
where $q$ and $k_j$ are positive integers for $1 \leq j \leq q$, $M_0 \in Sp^{*}_{1}(2h,\R)$ with $h=n-\Sigma_{j=1}^{q} k_j$ and each $N_{k_j}(\bb_j)$ is the normal form of eigenvalue $1$ given by \eqref{N1} or \eqref{Nk}.
\end{proposition}

We can apply Proposition \ref{standard form} to the degenerate path $\Phi$ ($\Phi(1)$ has eigenvalue $1$), then  there exists $P \in Sp(2n,\R)$ such that
\begin{align*}
	P \Phi(1) P^{-1} = N_{k_1}(\bb_1) \diamond N_{k_2}(\bb_2) \diamond \cdots \diamond N_{k_q}(\bb_q) \diamond M_0.
\end{align*} 
For any $(s, t) \in [-1,1] \times [0,1]$, define the paths
\begin{align}\label{rotational perturbation}
	\Phi(s,t) = \Phi(t) P^{-1}(e^{s p(t)\theta_0 J_{k_1}} \diamond e^{s p(t)\theta_0 J_{k_2}} \diamond \cdots \diamond e^{s p(t)\theta_0 J_{k_q}} \diamond I_{2h}) P,
\end{align}  
where $\theta_0 > 0$, $p(t)=0$ for $0 \leq t \leq t_0 \leq 1$, $\dot{p}(t) \geq 0$ for $0 \leq t \leq 1$, $\dot{p}(1) = 0$, $p(1)=1$ and $J_k=
      \begin{pmatrix}
		0 & -I_k \\
		I_k & 0 \\
		\end{pmatrix}$
is the standard symplectic matrix. When $t_0$ is sufficiently close to $1$, $\Phi(s,t)$  converges to $\Phi(t)$ as $s \to 0$ and satisfies
\begin{equation*}
	\left\{
	\begin{array}{ll}
		\Phi(0,t)= \Phi(t), \\
		\Phi(s, t)= \Phi(t) \ \text{for} \ \forall t \in [0,t_0],\ s \in [-1,1] ,\\
		\Phi(s, t) \ \text{is sufficiently close to} \ \Phi(1) \ \text{for} \ \forall t \in [t_0,1],\ s \in [-1,1]. \\ 
	\end{array} \right.	
\end{equation*}
Fix $s \not=0$, $\Phi(s,\cdot)$ is a non-degenerate path which has the Conley-Zehnder index. Then the Long's Maslov-type index can be given as follows
\begin{definition}[Long] \label{Def L}
	For any degenerate path $\Phi$ and sufficiently small $s>0$, the Long index is defined by
	\begin{align}\label{def-muL}
		\mu_{L}(\Phi):  = \mu_{CZ}\big(\Phi(-s,\cdot)\big)  =\inf \{\mu_{CZ}(\Psi) |\ \Psi\ \mathrm{is\ close\  to}\ \Phi\},
	\end{align}
	where $\Psi$ is any non-degenerate symplectic path that is sufficiently close to $\Phi$.
\end{definition}
The above definition can be found in Corollary 6.1.12 and Definition 6.1.13 in Long's book \cite{L}, which shows that the two ways of definition above, $i.e.$ CZ-index of specific rotation perturbation and taking infimum in all CZ-indices of nearby non-degenerate paths, actually coincide.
This definition can be applied to all symplectic paths as \eqref{3.2} and we call it Conley-Zehnder-Long index. The construction of this index is based on the method of rotational perturbation of a symplectic path, which is then deformed into some non-degenerate paths with the well-defined Conley-Zehnder index. The Conley-Zehnder-Long index is then defined as the infimum of these Conley-Zehnder indices. To illustrate this index, we present the following example:
\begin{example} \label{Eg L}
	\begin{align*}
		\Phi(t) = \begin{pmatrix}
			\cos 2\pi t & -\sin 2\pi t \\
			\sin 2\pi t & \cos 2\pi t \\
		\end{pmatrix},\ 0 \leq t \leq 1.
	\end{align*}
	Since $\Phi(1)=I_2$, then we can choose $P=I_2$ and obtain
	\begin{align*}
		\Phi(s, t) = \Phi(t)e^{sp(t)\theta J_1}
		=\begin{pmatrix}
			\cos (2\pi t + sp(t)\theta_0) & -\sin (2\pi t + sp(t)\theta_0) \\
			\sin (2\pi t + sp(t)\theta_0) & \cos (2\pi t + sp(t)\theta_0) \\
		\end{pmatrix},\ 0 \leq t \leq 1.
	\end{align*}
When $s>0$, the first kind eigenvalue of $\Psi(s,1)$ is positive, then $\mu_{CZ}(\Phi(s,\cdot))=3$. 
When $s<0$, the first kind eigenvalue of $\Psi(s,1)$ is negative, then $\mu_{CZ}(\Phi(s,\cdot))=1$. 
Then $\mu_{L}(\Phi) = \inf \{\mu_{CZ}(\Phi(s,\cdot))\}=1$.
\end{example}

\subsection{$L_0$-index}\label{rev L0 index}	
The second definition is $L_0$-index defined by Liu\cite{Liu1}, where $L_0=\{0\} \times \R^{n}$.
Let $\Phi$ be a symplectic path starting at identity $I$, denote it by
\begin{align*}
	\Phi(t) = \begin{pmatrix}
		S(t) & V(t) \\
		T(t) & U(t) \\
	\end{pmatrix},\ 0 \leq t \leq 1.
\end{align*}
The $n$ vectors from the column of $\begin{pmatrix}
	V(t) \\
	U(t)
\end{pmatrix}$ span a Lagrangian subspaces of $(\R^{2n}, \omega_0)$, which is $L_0$ when $t=0$. Denote the two connected components of $Sp(2n,\R)$ by
\begin{align*}
	Sp^{\pm}_{L_0}(2n,\R)=\{M \in Sp(2n,\R)) \ | \ \pm \det V_{M} >0\}
\end{align*}
for $M= \begin{pmatrix}
	S_{M} & V_{M} \\
	T_{M} & U_{M} \\
\end{pmatrix}$.
If $\Phi(1) \in Sp^{\pm}_{L_0}(2n,\R)$, then $\Phi$ is called a $L_0$-nondegenerate path. Otherwise, it is called an $L_0$-degenerate path.
To construct the $L_0$-index, Liu defined complex matrix function
\begin{align*}
	\bar{\rho}(\Phi(t))=[U(t)-iV(t)][U(t)+iV(t)]^{-1}.
\end{align*}
Then we can choose a continuous function $\bar{\Delta}(\Phi) : [0,1] \to \R$ such that
\begin{align}\label{3.5}
	\det \bar{\rho}(\Phi(t))=e^{2\pi i \bar{\Delta}(\Phi(t))}.
\end{align}
For any $L_0$-nondegenerate path $\Phi$, we consider the extensions of $\Phi$. Firstly, define
\begin{align*}
	\mathcal{E}_0(t)=I_{2n}\cos \frac{\pi (1-t)}{2} + J_n\sin \frac{\pi (1-t)}{2},\ 0 \leq t \leq 1
\end{align*}
as the extensions for $\Phi(0)$, where
\begin{align}\label{3.6}
	J_n=\begin{pmatrix}
		0 & -I_n \\
		I_n & 0 \\
	\end{pmatrix}.
\end{align}
As for $\Phi(1)$, if $\Phi(1) \in Sp^{+}_{L_0}(2n,\R)$, then we connect it to $J$ in $Sp^{+}_{L_0}(2n,\R)$. If $\Phi(1) \in Sp^{-}_{L_0}(2n,\R)$, then we connect it to $\begin{pmatrix}
	0 & D_n \\
	-D_n & 0 \\
\end{pmatrix}$
in $Sp^{-}_{L_0}(2n,\R)$, where $D_n$ is the diagonal matrix diag$\{-1,1,...,1\}$. Denote this extension for $\Phi(1)$ by $\mathcal{E}_1$, then the $L_0$-index for $L_0$-nondegenerate path is defined by
\begin{align}\label{3.7}
	i_{L_0}(\Phi)=\bar{\Delta}\big(\mathcal{E}_0\ \#\ \Phi\ \#\ \mathcal{E}_1(1)\big)-\bar{\Delta}\big(\mathcal{E}_0\ \#\ \Phi\ \#\ \mathcal{E}_1(0)\big).
\end{align}
For the symplectic paths $\Phi$ starting at $I$, the $L_0$-index is defined by
\begin{definition}\label{L0 index from I}
	\begin{align*}
		i_{L_0}(\Phi)= \inf \{i_{L_0}(\Psi) | \Psi \ \text{is  the $L_0$-nondegenerate path and sufficiently close to} \ \Phi \}.
	\end{align*}
\end{definition}
According to \cite{Liu1} Definition 4.3, the relationship between $i_{L_0}(\Phi)$ and $\mu_{L}(\Phi)$ is given by the concavity of $\Phi$ and denoted by
\begin{align}\label{3.9}
	c_{L_0}(\Phi)=\mu_{L}(\Phi) - i_{L_0}(\Phi)
\end{align}
By the Theorem 4.5 of \cite{Liu1}, the concavity only depends on the end point $\Phi(1)$ and the Lagrangian subspace $L_0$. We simply denote it by
\begin{align}\label{3.10}
	c(\Phi(1)):=c_{L_0}(\Phi)
\end{align}
and call it the $L_0$-concavity of $\Phi(1)$.

The idea of constructing $L_0$-index is similar to the Conley-Zehnder-Long index whose point of penetration is dealing with the symplectic paths starting at $I$.
Then the index of general symplectic paths (looked as symplectic path segments) can be defined by the index of these special symplectic paths.
Note that $L_0$-index depends on a Lagrangian subspace, for other Lagrangian subspace $L$, one can also define $L$-index \cite{Liu1}.

\subsection{The index of symplectic path segment (SPS index)}\label{rev SPS index}

Here we introduce the``SPS index"\cite{L0,L,Liu2} based on the Conley-Zehnder-Long index and Liu's $L_0$-index. Their definitions are
\begin{definition}\label{SPS index}
	For a general symplectic path $\Phi$, Long's SPS index is defined by
	\begin{align}\label{Long SPS}
		\hat{\mu}_{L}(\Phi)=\mu_{L}(\Phi' \# \Phi)-\mu_{L}(\Phi'),
	\end{align}
	Liu's SPS index is defined by
	\begin{align}\label{Liu SPS}
		\hat{i}_{L_0}(\Phi)=i_{L_0}(\Phi' \# \Phi)-i_{L_0}(\Phi'),
	\end{align}
	where $\Phi'$ is the symplectic path which starts at $I$ and ends at $\Phi(0)$, $\Phi' \# \Phi$ is the catenation of $\Phi'$ and $\Phi$ (see \eqref{catenation}), these indices are independent of the choices of $\Phi'$ and hence they are well defined \cite{L0,Liu2}.
\end{definition}

For a special constant identity path $\Phi(t)\equiv I_{2n}$, we have $\mu_{L}(I_{2n})=i_{L_0}(I_{2n})=-n$.
For a general symplectic path $\Phi$ starting from $\Phi(0)=I_{2n}$, it can be regarded not only as a standard symplectic path but also a symplectic path segment.
Both ordinary indices and SPS indices can be defined for it. Then the following equalities hold
\begin{align}\label{3.8} 	
	& \hat{\mu}_{L}(\Phi)=\mu_{L}(\Phi)+n,\\
	& \hat{i}_{L_0}(\Phi)=i_{L_0}(\Phi)+n.
\end{align}
\begin{remark}\label{rk L}	
	The Conley-Zehnder-Long index, Liu's $L_0$-index and their SPS index are different. For example, let
	\begin{align*}
		\Phi(t):=-\Phi_1(t)=I_{2n}\cos \frac{\pi t}{2} + J_n\sin \frac{\pi t}{2},\ 0 \leq t \leq 1.
	\end{align*}
	We have $i_{L_0}(\Phi)=0$ and $\mu_{L}(\Phi)=n$, then $\hat{i}_{L_0}(\Phi)=n$ and $\hat{\mu}_{L}(\Phi)=2n$.
\end{remark}
\begin{remark}\label{rk Cate}	
The Catenation axiom formulated in the Corollary 10 on page 148 of \cite{L} holds only for SPS index,
$e.g.$ for $\Phi \in C([a, c], Sp(2n,\R))$ with $a< b< c$, one has
$$\hat{\mu}_{L}(\Phi) = \hat{\mu}_{L}(\Phi |_{[a,b]}) + \hat{\mu}_{L}(\Phi |_{[b,c]}).$$
There is no such kind of Catenation axiom for indices $\mu_{L}(\Phi)$ and $i_{L_0}(\Phi)$
since they were not constructed before for those general symplectic paths.
\end{remark}

\subsection{Cappell-Lee-Miller index}\label{subsec-CLM}	
The fourth definition is the geometric definition of Cappell, Lee and Miller \cite{CLM}. It is the intersection number of a Lagrangian path with the codimension-one Maslov cycle, based on the proper-path index $\mu_{proper}$ of Guillemin and Sternberg \cite{GS}.
When $(V,\omega)=(\mathbb{R}^{2n},\omega_0)$, denote the Lagrangian Grassmannian by $\mathscr{L}(n):=\mathscr{L}(V)$. Let
\begin{align*}
	f(t)=(L_1(t),L_2(t)), \ 0 \leq t \leq 1
\end{align*}
be the pair of two smooth Lagrangian paths in $\mathscr{L}(n)$. $f(t)$ is called a proper path if
\begin{align*}
	L_1(t) \cap L_2(t) = \{0\}, \ t=0,1.
\end{align*}
For a general path, it can always become a proper path via perturbation. According to the Lemma 2.1 of \cite{CLM}, one has the following
\begin{proposition} \label{Pro CLM}
	Let $L, L' \in \mathscr{L}(n)$. Then $e^{\theta J_n}L' \in \mathscr{L}(n)$ for all $\theta$ and there exists an $\varepsilon$, $0<\varepsilon < \pi$, such that
	\begin{align*}
		L \cap e^{\theta J_n}L'=\{0\},\ \forall\ \theta\  \mathrm{with}\ 0< |\theta| < \varepsilon,
	\end{align*}
	where $J_n$ is given by \eqref{3.6}.
\end{proposition}
One can define the Maslov cycle for  $L \in \mathscr{L}(n)$ as
\begin{align}\label{3.11}
	\Sigma(L)=\{L' \in \mathscr{L}(n) | \dim(L' \cap L) \geq 1\}.
\end{align}
Then there exists a  $\theta $ with $0< \theta < \varepsilon$ such that
\begin{align*}
	f_{\theta}(t)=(L_1(t),e^{-\theta J_n}L_2(t)), \ 0 \leq t \leq 1,
\end{align*}
is a proper path and $\ e^{-\theta J_n}L_2(t)$ intersects $\Sigma(L_1(t))$ only at points of the top smooth stratum and crosses them transversally.
Then one can count the geometric intersection number with sign of  $e^{-\theta J_n}L_2(t)$ and $\Sigma(L_1(t))$.  Then the definition of Cappell-Lee-Miller index  is 
\begin{definition}\label{Def-CLM}
	\begin{align*}
		\mu_{CLM}(f):= \mu_{proper}(f_{\theta}),
	\end{align*}
which is the geometric intersection number, counted with signs, of the perturbed path  $\ e^{-\theta J_n}L_2(t)$  with the top stratum of $\Sigma(L_1(t))$ \cite{GS}.
\end{definition}
This number with sign depends on the orientation of $e^{-\theta J_n}L_2(t)$, for example, one can suppose
\begin{align*}
	f(t)=\big(\mathbb{R}\{1\},\  \mathbb{R}\{e^{i(t-\frac{1}{2})}\}\big), \ 0 \leq t \leq 1.
\end{align*}
Then $\mu_{CLM}(f)$ is the intersection number of two Lagrangian paths and we have
\begin{align*}
	\mu_{CLM}(f([0,\frac{1}{2}]))=0, \ \ \   \mu_{CLM}(f([\frac{1}{2},1]))=1,\ \ \ \mu_{CLM}(f([0,1]))=1.
\end{align*}
This means that  the perturbation fixes an orientation of the crossing. In this example,  $\mathbb{R}\{e^{i(t-\frac{1}{2})}\}$ crossing $\mathbb{R}\{1\}$ anti-clockwisely will be counted with the positive sign.

According to the system of axioms of \cite{CLM}, $\mu_{CLM}(f)$ is of symplectic invariance, i.e, for a symplectic transformation
$\Phi$, $\mu_{CLM}(\Phi L_1,\Phi L_2)=\mu_{CLM}(f)$, which concludes that one can choose $\Phi$ such that $\Phi L_1=\mathbb{R}^{n} \times \{0\}$.
Then we can fix $L_1(t)=\mathbb{R}^{n} \times \{0\}$.
Set $$d=\dim(L_1(0) \cap L_2(0)), \ l=\dim(L_1(1) \cap L_2(1)).$$
Define two tails for $f(t)$ as following
\begin{align*}
	&f_1(t)=(\mathbb{R}^{n}, e^{\frac{i \pi}{4}(1-t)}\mathbb{R}^{d} \oplus \mathbb{R}^{n-d}), \ 0 \leq t \leq 1 ,\\
	&f_2(t)=(\mathbb{R}^{n}, e^{-\frac{i \pi t}{4}}\mathbb{R}^{l} \oplus \mathbb{R}^{n-l}), \ 0 \leq t \leq 1.
\end{align*}
Let
\begin{equation}\label{bar f}
	\bar{f}(t)) =f_1 \# f \# f_2 (t)= \left\{
	\begin{array}{ll}
		f_1(3t), & 0 \leq t < \frac{1}{3}, \\
		f(3t-1), & \frac{1}{3} \leq t < \frac{2}{3}, \\
		f_2(3t-2), & \frac{2}{3} \leq t \le 1,
	\end{array} \right.
\end{equation}
 and then $\bar{f}(t))$ is a proper path.
At the intersection time $t_j$, $\bar{f}(t)$ is locally isomorphic to one of the following two cases:
\begin{align*}
	(\mathbb{R}^{n}, e^{i(t-t_j)}\mathbb{R}^{1} \oplus \mathbb{R}^{n-1}) \ \text{or} \ (\mathbb{R}^{n}, e^{-i(t-t_j)}\mathbb{R}^{1} \oplus \mathbb{R}^{n-1}),\ |t-t_j|<\delta.
\end{align*}
Suppose that there are $p$ intersection points and $q$ intersection points of these two cases, then $\mu_{CLM}(\bar{f})= p-q$. By Definition \ref{Def-CLM}, we have $\mu_{CLM}(f_1)=-d$ and $\mu_{CLM}(f_2)=0$. According to the system of axioms by \cite{CLM}, $\mu_{CLM}(\bar{f})= \mu_{CLM}(f_1)+\mu_{CLM}(f)+\mu_{CLM}(f_2)$. Thus, we can obtain
\begin{align}\label{3.12}
	\mu_{CLM}(f)= d+p-q.
\end{align}
In addition, the Cappell-Lee-Miller index for Lagrangian path pairs can naturally induce an index for general symplectic paths as follows
\begin{definition}\label{Def-CLM-sp}
Let $\Phi$ be a general symplectic path and put $L_1=\mathbb{R}^{n} \times \{0\}$, $f(t)=(L_1, \Phi(t) L_1)$.
Then the Cappell-Lee-Miller index for general symplectic path $\Phi$ is defined as
	\begin{align}\label{CLM-ind-path}
		\mu_{CLM}(\Phi):=\mu_{CLM}(f).
	\end{align}
\end{definition}

\begin{remark}[Another expression of the index $\mu_{CLM}$]\label{CM-integ}
Recently, inspired by the construction of index $\mu$ in this paper, Cappell and Miller \cite{CM} provided another  expression of the index $\mu_{CLM}$ in terms of integrals plus end point corrections. 
See Theorem 1.3 of \cite{CM}. 
\end{remark}

For $L_1=\mathbb{R}^{n} \times \{0\}$ and a given symplectic path $\Phi$, there always exists a special orthogonal symplectic path $O$ such that one can compute the index of $\Phi$ via $O$.
This is the following
\begin{proposition}\label{sp to osp}
	Fix $L_1=\mathbb{R}^{n} \times \{0\}$, for any symplectic path, written  as
	\begin{align*}
		\Phi(t)=\begin{pmatrix}
			S(t) & U(t) \\
			T(t) & V(t)
		\end{pmatrix},
	\end{align*}
	there exists a corresponding orthogonal symplectic path as
	$$
    O(t)=\begin{pmatrix}
			\cos \Theta(t) & -\sin \Theta(t) \\
			\sin \Theta(t) & \cos \Theta(t)
		\end{pmatrix},
    $$
where
		$$\Theta(t)={\rm diag} \{\theta_1(t), \theta_2(t), \cdots, \theta_n(t)\},\ \theta_j(0) \in [0,\pi), \ 1 \leq j \leq n,$$
	such that
$$\mu_{CLM}(\Phi)=\mu_{CLM}(O).$$
	Those functions $\theta_j :[0,1] \to \R (j=1,2,\cdots,n)$ above are continuous and determined by $\Phi$.
	Moreover, $O(t)L_1 \in \Sigma(L_1)$ if and only if $O(t)$ has eigenvalues equal to $1$ or $-1$, where $\Sigma(L_1)$ is the Maslov cycle (see \eqref{3.11}).
\end{proposition}

\begin{proof}
	Set $L_2(t)=\Phi(t)L_1=\{(S(t)x,T(t)x) \ | \ x \in \R^{n}\}$ and simply write $L_2= \text{im} \begin{pmatrix}
		S \\
		T
	\end{pmatrix} $. By Definition \ref{Def symp.matrix}, $S^{T} T$ is a path of symmetric matrices and
	$\begin{pmatrix}
			S \\
			T
		\end{pmatrix} $
	is a Lagrangian frame (see \eqref{3.13} below).  Let 
	$$P=S^{T}S+T^{T}T,$$
	then $P$ is a positive definite symmetric path. Define
	$$X=S P^{-\frac{1}{2}}, \ Y=T P^{-\frac{1}{2}},$$
	we have
	 $$X^{T}Y=(P^{-\frac{1}{2}})^{T} S^{T}T P^{-\frac{1}{2}}=(P^{-\frac{1}{2}})^{T} T^{T}S P^{-\frac{1}{2}}=Y^{T}X,$$ $$X^{T} X+Y^{T} Y= (P^{-\frac{1}{2}})^{T} (S^{T}S+T^{T}T) P^{-\frac{1}{2}} = I.$$
	Hence $O_0 := \begin{pmatrix}
		X & -Y \\
		Y & X
	\end{pmatrix} $ is an orthogonal symplectic path and 
	$$L_2= \text{im} \begin{pmatrix}
		S \\
		T
	\end{pmatrix} = \text{im} \begin{pmatrix}
	X \\
	Y
	\end{pmatrix} = O_0 L_1.$$
	
	Let $U_0=X+iY$, then $U_0$ is an unitary path and $U_0 U_0^{T}$ is a symmetric unitary path. By Corollary 2.5.18 of \cite{HJ}, there exists a real orthogonal path $Q$ such that
	$$Q^{T} U_0 U_0^{T} Q =  e^{2i \Theta},$$
	where $\Theta={\rm diag} \{\theta_1, \theta_2, \cdots, \theta_n\}, \lambda_{j}=\cos 2\theta_j + i \sin 2\theta_j (1 \leq j \leq n)$ is the eigenvalue path of $U_0 U_0^{T}$ and $\theta_j $ is unique when we choose $\theta_j (0) \in [0, \pi)$. Let $R=e^{-i \Theta} Q^{T} U_0$, then $R$ is an unitary path and
	$$\bar{R}=e^{i \Theta} Q^{T} \bar{U_0}=(e^{-i \Theta} Q^{T} U_0 U_0^{T} Q )Q^{T} \bar{U_0} = e^{-i \Theta} Q^{T} U_0 = R.$$
	Thus, $R$ is an orthogonal path. Since $U_0=Qe^{i \Theta} R$, we can obtain
	$$X=Q \cos \Theta R, \ Y= Q \sin \Theta R $$
	and then 
	$$L_2= \text{im} \begin{pmatrix}
		X \\
		Y
	\end{pmatrix} = \text{im} \begin{pmatrix}
	Q \cos \Theta R\\
	Q \sin \Theta R
	\end{pmatrix} = {\rm diag}\{Q,Q\} O L_1.$$
	Set $W={\rm diag}\{Q,Q\}$, then $W$ is an orthogonal symplectic path.  Since $\mu_{CLM}$ is a symplectic invariant and by Definition \ref{Def-CLM-sp}, we have
	\begin{align*}
		\mu_{CLM}(\Phi)=\mu_{CLM}(L_1,L_2)=\mu_{CLM}(W^{-1} L_1,W^{-1}L_2)=\mu_{CLM}(L_1,OL_1)=\mu_{CLM}(O).
	\end{align*}
	In addition, $O(t)L_1 \in \Sigma(L_1)$ if and only if $\sin \Theta(t) x=0$ has at least $1$-dimensional solution space, if and only if $\det \sin \Theta(t) = \prod_{j=1}^{n} \sin \theta_j(t) = 0$. It is equivalent to $\theta_j(t)=k \pi$ for $k \in \mathbb{Z}$ and some $j$, which means that $O(t)$ has eigenvalues equal to $1$ or $-1$. This completes the proof.
\end{proof}

\subsection{Robbin-Salamon index}\label{subsec-RS}	
The fifth version of definition is given by Robbin and Salamon \cite{RS1}, which defines the the Maslov index of the pair of Lagrangian paths via the crossing form and uses it to construct the Maslov index for symplectic paths.

Let $L_1(t), L_2(t) \in \mathscr{L}(n)$ be two smooth paths and we first consider $L_2(t)$ is a constant path, denoted by $L_2$. The Langrangian complement of $L_1(t)$ is denoted by $L_1^{c}(t)$ (i.e. $L_1 \oplus L_1^{c}=\mathbb{R}^{2n}$ for $\forall t$). By Theorem 1.1 of \cite{RS1}, for $\forall\ v \in L_1(t_0)$ and $t$ close to $t_0$, choose $l(t) \in L_1^{c}(t_0)$  such that $v+l(t) \in L_1(t)$. Then the form
\begin{align*}
	Q(v):= -\frac{d}{dt}\Big|_{t=t_0} \omega_0(v,l(t))
\end{align*}
is well defined. We can express this form by Lagrangian frame of $L \in \mathscr{L}(n)$, which is an injective linear map $Z: \mathbb{R}^{n} \to \mathbb{R}^{2n}$ whose image is $L$.
It has the form as
\begin{align}\label{3.13}
	Z=	\begin{pmatrix}
		X \\
		Y
	\end{pmatrix} ,\
	X^{T}Y = Y^{T}X,
\end{align}
where $X,Y$ are $n \times n$ matrices. Let $Z(t)$ be the Lagrangian frame of $L_1(t)$, then the form
\begin{align}\label{3.14}
	Q(v)= \langle X(t_0)u, \dot{Y}(t_0)u \rangle - \langle Y(t_0)u, \dot{X}(t_0)u \rangle,
\end{align}
where $\langle . \ , \ . \rangle$ is the Euclidean inner product, $\dot{X}(t_0),\dot{Y}(t_0)$ are the differential on $t_0$ and $v=Z(t_0)u$.
Then the crossing form is defined by
\begin{align}\label{3.15}
	\Gamma(L_1,L_2,t)(v)=Q(v)|_{L_1(t) \cap L_2}.
\end{align}
The crossing form is a quadratic form and  $\Gamma(L_1,L_2,t)(v)=0$ holds unless at the time $t$ when $L_1(t)$ crosses the Maslov cycle $\Sigma(L_2)$ transversally (regular crossings). Then the index of $(L_1(t),L_2)$ having only regular crossings is defined as
\begin{align}\label{3.16}
	\mu_{RS}(L_1,L_2)=\frac{1}{2}\sum_{t=0,1}\text{Sign} \Gamma(L_1,L_2,t) + \sum_{0<t<1}\text{Sign} \Gamma(L_1,L_2,t),
\end{align}
where $\text{Sign} \Gamma(L_1,L_2,t)$ is the difference of the positive exponential inertial and the negative exponential inertial of $\Gamma(L_1(t),L_2,t)(v)$. More generally, $\mu_{RS}(L_1,L_2)$ can be defined for continuous path. Every continuous path is homotopic with fixed endpoints to one having only regular crossings and then they have the same Maslov index.
For a symplectic path $\Phi : [0,1] \to Sp(2n,\mathbb{R})$, let $L_2=\{0\} \times \mathbb{R}^{n}$ and $L_1(t)=\Phi(t)L_2$, the Maslov index for $\Phi$ is
\begin{align}\label{3.17}
	\mu_{RS}(\Phi):=\mu_{RS}(L_1,L_2).
\end{align}
For a general path
\begin{align*}
	f(t)=(L_1(t),L_2(t)), \ 0 \leq t \leq 1,
\end{align*}
one can define the relative crossing form
\begin{align}\label{3.18}
	\Gamma{'}(f,t)=\Gamma(L_1,L_2(t),t)-\Gamma(L_2,L_1(t),t).
\end{align}
Then the relative Maslov index is similarly defined by
\begin{align}\label{3.19}
	\mu'_{RS}(f)=\frac{1}{2}\sum_{t=0,1}\text{Sign} \Gamma{'}(f,t) + \sum_{0<t<1}\text{Sign} \Gamma{'}(f,t).
\end{align}
By the Theorem 2.3 of \cite{RS1}, $\mu_{RS}(L_1,L_2)$ and $\mu{'}_{RS}(L_1,L_2)$ are of symplectic invariant and they also satisfy the system of axioms of \cite{CLM}.
\begin{remark}\label{CZ & RS}
	According to \cite{RS1}, for any non-degenerate symplectic path $\Phi \in \mathcal{P}(2n,\R)$ with $\Phi(0)=I_{2n}$, the Conley-Zehnder index and the Robbin-Salamon index have the relationship
	\begin{align*}
		\mu_{CZ}(\Phi)=\mu_{RS}\big(Gr(\Phi),Gr(I_{2n})\big),
	\end{align*}
	where $Gr(M):=\{(x,Mx) \ | x \in \R^{2n}\}$ is the graph of $M \in Sp(2n,\R)$, which is viewed as a Lagrangian subspace of $(\R^{2n} \times \R^{2n}, -\omega_0 \oplus \omega_0 )$.
	This is a realization of Conley-Zehnder index for non-degenerate symplectic paths from the point of view of Robbin-Salamon index for Lagrangian paths.
\end{remark}

For the convenience of calculating the $\mu_{RS}$, we still convert it into an index for an orthogonal symplectic path, similar to the $\mu_{CLM}$. That is
\begin{proposition}\label{RS:sp to osp}
	 Fix $L_2=\{0\} \times \mathbb{R}^{n}$, for any symplectic path, written  as
	\begin{align*}
		\Phi(t)=\begin{pmatrix}
			S(t) & U(t) \\
			T(t) & V(t)
		\end{pmatrix},
	\end{align*}
	there exists a corresponding orthogonal symplectic path $O$ as the form in Proposition \ref{sp to osp} such that
	$$\mu_{RS}(\Phi)=\mu_{RS}(O),$$
	and $O(t)L_2 \in \Sigma(L_2)$ if and only if $O(t)$ has eigenvalues equal to $1$ or $-1$, where $\Sigma(L_2)$ is the Maslov cycle.
\end{proposition}
\begin{proof}
The argument process is similar to that of Proposition \ref{sp to osp}, it is merely a matter of replacing $S,T$ with $U,V$, respectively.	
\end{proof}

\subsection{Leray-Gosson index}\label{subsec-LG}

The last one of Maslov indices in this Section is the Leray-Gosson index. In \cite{Go}, Gosson showed that Leray's index \cite{Le} for Lagrangian paths already contains the information needed to recover the usual Lagrangian and symplectic intersection indices. In this subsection, we recall this construction. 

Fix the base point $L_* \in \mathscr{L}(n)$ and set
\begin{equation}\label{LG cover}
\widetilde{\mathscr{L}(n)}
:=\bigl\{[L]\ \big|\
L:[0,1]\to\mathscr{L}(n),\ L(0)=L_*\bigr\}/\sim,
\end{equation}
where $\sim$ denotes homotopy with fixed endpoints. The projection
\[
\pi_{\mathscr{L}}:\widetilde{\mathscr{L}(n)}\to\mathscr{L}(n),\qquad
\pi_{\mathscr{L}}([L])=L(1),
\]
is the universal covering of $\mathscr{L}(n)$. To see that it is simply connected, note that a loop $\gamma$ in $\mathscr{L}(n)$based at $L_*$ lifts to a loop in $\widetilde{\mathscr{L}(n)}$ based at $\widetilde{L_*}$ exactly when $[\gamma]=\widetilde{L_*}$, i.e. when $\gamma$ is null-homotopic. For $\tilde L\in\widetilde{\mathscr{L}(n)}$, we write $L=\pi_{\mathscr{L}}(\tilde L)$. For a Lagrangian path $L:[0,1]\to\mathscr{L}(n)$, a lift is denoted by $\widetilde{L(t)}$, which is determined by the initial lift $\widetilde{L(0)}$. If $L$ starts at $L_*$, its terminal lift is $[L]$.

Let \(L_1,L_2,L_3\) be pairwise transversal. For \(z_i\in L_i\) with \(z_1+z_2+z_3=0\), put
\[
q(z_1)=\omega_0(z_2,z_3),
\]
$q(z_2)=\omega_0(z_3,z_1)$ and $q(z_3)=\omega_0(z_1,z_2)$. These forms have the same index of inertia \cite{Le}, then we have the following definitions:

\begin{definition}[Inertia index]\label{Def Inert}
For a pairwise transversal triple $(L_1,L_2,L_3)$, the inertia index of any one of the equivalent quadratic forms $q$ on $L_1,L_2,L_3$, is denoted by
\[
\operatorname{Inert}(L_1,L_2,L_3).
\]
It can be the nunber of positive or negative engenvalues of $q$.
\end{definition}

\begin{definition}[Leray index]\label{Def Leray index}
The Leray index is a function
\[
\mu_{Le}:\{(\tilde L,\tilde L')\mid L\cap L'=0\}\to\Z
\]
which is locally constant on its domain and satisfies
\begin{equation}\label{LG Leray identity}
\mu_{Le}(\tilde L_1,\tilde L_2)
-\mu_{Le}(\tilde L_1,\tilde L_3)
+\mu_{Le}(\tilde L_2,\tilde L_3)
=\operatorname{Inert}(L_1,L_2,L_3)
\end{equation}
for every pairwise transversal triple $(L_1,L_2,L_3)$.
\end{definition}

Such a function is unique. Specifically, if $\mu'_{Le}$ is another function with the same properties, put $\delta=\mu_{Le}-\mu'_{Le}$. Subtracting the two identities \eqref{LG Leray identity} gives
\[
\delta(\tilde L_1,\tilde L_2)
-\delta(\tilde L_1,\tilde L_3)
+\delta(\tilde L_2,\tilde L_3)=0.
\]
For any $\tilde L_3$ transversal to both $\tilde L_1,\tilde L_2$, we have
\[
\delta(\tilde L_1,\tilde L_2)
=\delta(\tilde L_1,\tilde L_3)-\delta(\tilde L_2,\tilde L_3).
\]
The right-hand side does not depend on the choice of $\tilde L_3$. Since $\delta=\mu_{Le}-\mu'_{Le}$ is locally constant on the domain (with transversal condition), then $\delta$ extends to a locally constant function on the connected space $\widetilde{\mathscr{L}(n)}\times\widetilde{\mathscr{L}(n)}$, hence it is constant and we obtian $\delta=0$.

To generalize the Leray index to the non transversal case, Gosson uses the Wall-Kashiwara signature  define by
\begin{equation}\label{LG tau}
\tau(L_1,L_2,L_3)
:=\operatorname{Sign} Q(z_1,z_2,z_3),  
\end{equation}
where the quadratic form 
\[
Q(z_1,z_2,z_3)=\omega_0(z_1,z_2)+\omega_0(z_2,z_3)+\omega_0(z_3,z_1)
\]
is defined on \(L_1\oplus L_2\oplus L_3\). When the triple is pairwise transversal,
\[
\tau(L_1,L_2,L_3)=2\operatorname{Inert}(L_1,L_2,L_3)-n,
\]
whether "$\operatorname{Inert}$" is defined by the positive or negative index in Defintion \ref{Def Inert}.

\begin{definition}[Generalized Leray index]\label{Def generalized Leray index}

\label{Def Leray index}
The generalized Leray index is a function
\[
\hat{\mu}_{Le}:\{(\tilde L,\tilde L') \in \widetilde{\mathscr{L}(n)}\times\widetilde{\mathscr{L}(n)}\}\to\Z
\]
which is locally constant on $\{(\tilde L,\tilde L')\mid L\cap L'=0\}$ and satisfies
\begin{equation}\label{LG Leray identity}
	\hat{\mu}_{Le}(\tilde L_1,\tilde L_2)
	-\hat{\mu}_{Le}(\tilde L_1,\tilde L_3)
	+\hat{\mu}_{Le}(\tilde L_2,\tilde L_3)
	=\tau(L_1,L_2,L_3)
\end{equation}
for every triple $(L_1,L_2,L_3)$. In particular, for transversal pairs, we have
\begin{equation}\label{LG hat mu transversal}
	\hat\mu_{Le}(\tilde L,\tilde L')=2\mu_{Le}(\tilde L,\tilde L')-n.
\end{equation}

\end{definition}

This is Gosson's generalized Leray index and its main properties are collected in Theorem 4 of \cite{Go}.
For general Lagrangian paths and symplectic paths, Gosson gives the following construction via the generalized Leray index.

\begin{definition}[Leray-Gosson index]\label{Def LG}
Let $L:[0,1]\to\mathscr{L}(n)$ be a Lagrangian path and $L_0 \in \mathscr{L}(n)$. Choose a lift $\widetilde{L(0)}$ of $L(0)$, and let $\widetilde{L(1)}$ be the terminal point of the lifted path. Define
\begin{equation}\label{LG Lagrangian}
\mu_{LG}(L,L_0)
:=
\frac12\Bigl[
\hat\mu_{Le}\bigl(\widetilde{L(1)},\tilde L_0 \bigr)
-
\hat\mu_{Le}\bigl(\widetilde{L(0)},\tilde L_0 \bigr)
\Bigr].
\end{equation}
For a symplectic path \(\Phi:[0,1]\to Sp(2n,\R)\), define
\begin{equation}\label{LG symplectic}
	\mu_{LG}(\Phi,L_0)
	:=
	\mu_{LG}(\Phi L_0,L_0).
\end{equation}
\end{definition}

In fact, \eqref{LG Lagrangian} is independent of the choice of  the lifts and hence it is well defined \cite{Go}. We call $\mu_{LG}$ the Leray-Gosson index and \eqref{LG Lagrangian}, \eqref{LG symplectic} are its Lagrangian and symplectic versions, respectively.
 
By Proposition 9 in \cite{Go}, Gosson shows that the index $\mu_{LG}$ and the Robbin-Salamon index (see Subsection \ref{subsec-RS}) are determined by the same axioms, and therefore they are equal. This means the following proposition holds.

\begin{proposition}\label{Prop LG=RS}
Let $L:[0,1]\to\mathscr{L}(n)$ be a continuous Lagrangian path and $L_0 \in\mathscr{L}(n)$, then
\begin{equation}\label{LG=RS Lagrangian}
\mu_{LG}(L,L_0)=\mu_{RS}(L,L_0).
\end{equation}
Consequently, for every continuous symplectic path $\Phi:[0,1]\to Sp(2n,\R)$, taking $L_0=\{0\}\times\R^n$, define
\begin{equation}\label{LG=RS symplectic}
\mu_{LG}(\Phi):=\mu_{LG}(\Phi,L_0),
\end{equation}
hence $\mu_{LG}(\Phi)=\mu_{RS}(\Phi)$, where $\mu_{RS}(\Phi)$ is defined by \eqref{3.17}.
\end{proposition}

\section{The construction of the Maslov type index}\label{Sec-Mas-def}

In this section, we combine the methods of constructing  $\mu_{CZ}$, $\mu_{L}$ and $\mu_{CLM}$ and
use the perturbation and extension argument to define an index in a consistent way, no matter whether the starting point of the path is identity or not.\\

\noindent{\bf Orthogonalization}

We first consider the orthogonalization for a path $\Phi \in \mathcal{P}(2n,\mathbb{R})$, since it is much easier to do calculation and operation for  orthogonal symplectic matrices. 
The idea of orthogonalization is that we can extend the endpoints to the orthogonal symplectic matrices meanwhile keeping the rotation number  invariant.
According to \cite{CZ2}, a symplectic matrix $M$ can be represented as the polar form
\begin{align}\label{4.2}
	M=PO,
\end{align}
where $P=(MM^T)^{1/2}$ is a positive definite symmetric and symplectic matrix and $O=P^{-1}M$ is an orthogonal symplectic matrix. For $\Phi \in \mathcal{P}(2n,\mathbb{R})$, by the representation above, we set
\begin{align*}
	\Phi(0)=P_{1}O_{1}, \ \Phi(1)=P_{2}O_{2}.
\end{align*}
Since the set of positive definite symplectic and symmetric matrices is contractible, we can find the positive definite symplectic and symmetric paths $P_{1}(t), P_{2}(t)$ such that
\begin{align*}
	P_{1}(0)=P_{1}, \ P_{2}(0)=P_{2}, \ P_{1}(1)=P_{2}(1)=I
\end{align*}
and define $\beta_1(t)=P_{1}(1-t)O_{1},\beta_2(t)=P_{2}(t)O_{2}$. 

If $\Phi(0)$ and $\Phi(1)$ are block diagonal and orthogonal symplectic matrices (or diagonal matrices), by Remark \ref{CZ rotation num}, we have
\begin{align}\label{rot-num hold}
	\Delta(\beta_1)=\Delta(\beta_2)=0.
\end{align}
This means that we can add the two tails $\beta_1$ and $\beta_2$ to $\Phi$ with these special endpoints and it will not change the rotation number.

In general case, by calculation on some simple example, one can see that the rotation number of such a path of polar form $P(t)O$ might not vanish.
Nevertheless, we can still choose two tails by the method of normalization of eigenvalues instead of using the polar form. We will do this as follows.
Suppose that $\Phi(0)$ and $\Phi(1)$ are arbitrary symmetric matrices, we can consider the  normalization of eigenvalues of $\Phi(0)$ and $\Phi(1)$.
For any symplectic matrix $M$, denote by $\sigma(M)=\{l_j e^{i \theta_j} \ | \ l_j>0, \theta_j \leq \theta_{j+1}, j=1,2,\dots,2n \}$ the set of all eigenvalues of $M$.
We want to construct a block diagonal and orthogonal symplectic matrix whose first kind eigenvalues on $\mathbb{S}^1$ are the same as $M$'s.
Denote the set of these first kind eigenvalues by $ \{e^{i \theta_j} \ | \ \theta_j \leq \theta_{j+1}, j=1,2,\dots,n  \}$. Each first kind eigenvalue $e^{i \theta_j}$ can correspond to the symplectic matrix
	\begin{align*}
		O_j=\begin{pmatrix}
			\cos \theta_j & -\sin \theta_j \\
			\sin \theta_j & \cos \theta_j
		\end{pmatrix} \in Sp(2,\mathbb{R}), \ j=1, 2, \dots, n.
	\end{align*}
	Then we construct a unique block diagonal and orthogonal symplectic matrix from $M$ as
	\begin{align}\label{block diagonal O}
		O=O_1 \diamond O_2 \diamond \cdots \diamond O_n \in Sp(2n,\mathbb{R}).
	\end{align}
	We can choose any symplectic path $\Psi(t)$ such that it starts at $M$ and ends at $O$. Since each first kind eigenvalue starts at $l_j e^{i\theta_j} (l>0)$ and ends at $e^{i\theta_j}$,
by the definition of $\rho$ (see \eqref{2.1}\eqref{2.2}) and \eqref{2.6}, at two end points of $\Psi$, we have
$$e^{i\pi \alpha(0)}=\rho(\Psi(0))=\rho(\Psi(1))=e^{i\pi \alpha(1)}.$$
Then $\Delta(\Psi)=\alpha(1)-\alpha(0)=2k$ $(k \in \mathbb{Z})$. If $k\neq 0$, $\Psi$ is the path we desired to achieve the normalization of eigenvalues.
Otherwise, one can construct a loop as
	$$
	\Psi'(t)=O'_1(t) \diamond O_2 \diamond \cdots \diamond O_n, \ O'_1(t)=\begin{pmatrix}
		\cos (\theta_1-2k\pi t) & -\sin (\theta_1-2k\pi t)\\
		\sin (\theta_1-2k\pi t) & \cos (\theta_1-2k\pi t)
	\end{pmatrix}	 .\vspace{5mm}
	$$
Then the catenation $\Psi \# \Psi'$ also starts at $M$ and ends at $O$, and satisfies $\Delta(\Psi \# \Psi')=2k-2k=0$. For $\Phi(0)$ and $\Phi(1)$, they are corresponding to two  block diagonal and orthogonal symplectic matrices $O_1$ and $O_2$. By the discussion above, we can choose the tails of $\Phi$ as $\beta_j:[0,1] \to Sp(2n,\mathbb{R}), \  j=1,2$ such that
	\begin{align*}
		\beta_1(0)=O_1,\ \beta_1(1)=\Phi(0),\ \beta_2(0)=\Phi(1),\ \beta_2(1)=O_2,
	\end{align*}
	satisfying \eqref{rot-num hold} $i.e.$ $\Delta(\beta_j)=0$. This means that we can also add the two tails $\beta_1$ and $\beta_2$ to $\Phi$ with the general endpoints and does not change the rotation number.
%\end{remark}

Then we define
\begin{equation}\label{4.3}
	\Phi^{\#}(t) := \left\{
	\begin{array}{ll}
		\beta_1(3t), & 0 \leq t < \frac{1}{3}, \\
		\Phi(3t-1), & \frac{1}{3} \leq t < \frac{2}{3}, \\
		\beta_2(3t-2), & \frac{2}{3} \leq t \le 1,
	\end{array} \right. 	
\end{equation}
and we call \eqref{4.3} the {\bf orthogonalization} of $\Phi$ at the two endpoints. By the construction as above, we have the following
\begin{lemma}\label{lemma orthogonalization}
	For $\Phi \in \mathcal{P}(2n,\mathbb{R})$, the rotation number \eqref{2.7} is independent of the choices of orthogonalizations, i.e., $\Delta(\Phi^{\#})=\Delta(\Phi)$.
\end{lemma}
\begin{proof}
	By \eqref{2.9} and $\Delta(\beta_1)=\Delta(\beta_2)=0$, the rotation number
	\begin{align}\label{orthogonal num hold}
		\Delta(\Phi^{\#})=\Delta(\beta_1)+\Delta(\Phi)+\Delta(\beta_2)=\Delta(\Phi),
	\end{align}
	which is determined by $\Phi$ and means the rotation number is independent of the choices of orthogonalizations.
\end{proof}
 After the orthogonalization, $\Phi^{\#}(0)$ and $\Phi^{\#}(1)$ are orthogonal and symplectic matrices, then all eigenvalues of $\Phi^{\#}(0)$ and $\Phi^{\#}(1)$ lie on the unit circle $\mathbb{S}^{1}\subset \C$.
 In particular, if $\Phi(0)$ and $\Phi(1)$ are already orthogonal symplectic matrices, then one can just set the orthogonalization as
\begin{equation}\label{4.4}
	\Phi^{\#}(t) := \left\{
	\begin{array}{ll}
		\Phi(0), & 0 \leq t < \frac{1}{3}, \\
		\Phi(3t-1), & \frac{1}{3} \leq t < \frac{2}{3}, \\
		\Phi(1), & \frac{2}{3} \leq t \le 1.
	\end{array} \right.
\end{equation}
As Remark \ref{rk cycle} claims, we define the degenerate cycle as
\begin{align}\label{4.5}
	\mathscr{C}(2n,\mathbb{R}) = Sp_{1}(2n,\mathbb{R}) \cup Sp_{-1}(2n,\mathbb{R}).
\end{align}
The motivation of considering such degenerate cycle is that we want to explain the index of Lagrangian pairs (like Cappell-Lee-Miller index) from the point of view of symplectic paths.
Based on such a consideration, we obtain the result of Theorem \ref{Thm5}. The principle of this result is that the cycle $\mathscr{C}(2n,\mathbb{R})$ can correspond to the Maslov cycle (see \eqref{3.11}).
If $O(t)$ is an orthogonal symplectic path as in Proposition \ref{sp to osp} and $L_1 = \mathbb{R}^{n} \times \{0\}$, then $O \in \mathscr{C}(2n,\mathbb{R})$ if and only if $OL_1 \in \Sigma(L)$,
where $\Sigma(L_1)$ is the Maslov cycle of $L_1$. This means that Capell-Lee-Miller index can be explained as  a special case (orthogonal symplectic path) of our definition.
\begin{remark}\label{rk cycle & Maslov cycle}
	Given an orthogonal symplectic path $O$ and $L_1 = \mathbb{R}^{n} \times \{0\}$, then it is corresponding to a Lagrangian pair $(L_1,OL_1)$. Conversely, for any $\Phi \in \mathcal{P}(2n,\mathbb{R})$ and $L_1 = \mathbb{R}^{n} \times \{0\}$, by Proposition \ref{sp to osp}, there exists a corresponding orthogonal symplectic path $O$. This means that any Lagrangian pair  $(L_1,\Phi L_1)$ can correspond to an orthogonal symplectic path $O$. In addition, if we only define the cycle as the component of eigenvalue $1$ (i.e. $Sp_1(2n,\R)$), for $\Phi \in \mathcal{P}(2n,\mathbb{R})$, we can consider the graph $Gr(\Phi)=\{(x,\Phi x) \ | x \in \R^{2n}\}$ as a Lagrangian subspace of $(\R^{2n} \times \R^{2n}, -\omega_0 \oplus \omega_0)$, then $\Phi(t) \in Sp_1(2n,\R)$ if and only if $Gr(\Phi(t)) \in \Sigma(Gr(I))$, this means that any $\Phi \in \mathcal{P}(2n,\mathbb{R})$ can correspond to a Lagrangian pair in higher dimension. But for any Lagrangian pair $(L_1,L_2)$ which are the Lagrangian subspaces of $(\R^{4k+2},\omega_0) (k \in \mathbb{Z})$, it can not be converted into symplectic paths because it requires that these symplectic paths take values in ``$Sp(2k+1,\R)$", which is impossible.
\end{remark}

According to the Lemma 3.2 of \cite{SZ}, $Sp^{*}_{1}(2n,\mathbb{R})$ (see \eqref{3.1}) has two connected components $Sp^{+}_{1}(2n,\mathbb{R})$ and $Sp^{-}_{1}(2n,\mathbb{R})$. Then $Sp^{*}_{-1}(2n,\mathbb{R})$ also has two connected components $Sp^{+}_{-1}(2n,\mathbb{R})$ and $Sp^{-}_{-1}(2n,\mathbb{R})$ because $A \mapsto -A$ is a homeomorphism between $Sp^{*}_{1}(2n,\mathbb{R})$ and $Sp^{*}_{-1}(2n,\mathbb{R})$. We define
\begin{align}\label{4.6}
	&\mathscr{N}(2n,\mathbb{R}):=Sp^{+}_{1}(2n,\mathbb{R}) \cap Sp^{+}_{-1}(2n,\mathbb{R}).
\end{align}
Denote by $\sigma_1(M)$ the set of the first kind eigenvalues of $M$, the two connected components of $\mathscr{N}(2n,\mathbb{R})$ can be denoted by
\begin{align}\label{4.7}
	&\mathscr{N}^{+}(2n,\mathbb{R})=\{M \in \mathscr{N}(2n,\mathbb{R}) \ | \  \prod_{\lambda \in \sigma_1(M) \cap \mathbb{S}^{1} \backslash \{\pm1\}} Im \lambda >0 \}, \\
	&\mathscr{N}^{-}(2n,\mathbb{R})=\{M \in \mathscr{N}(2n,\mathbb{R}) \ | \  \prod_{\lambda \in \sigma_1(M) \cap \mathbb{S}^{1} \backslash \{\pm1\}} Im \lambda <0 \}.\label{4.99}
\end{align}
We emphasize that these two connected components are different, i.e. $\mathscr{N}^{+}(2n,\mathbb{R}) \cap \mathscr{N}^{-}(2n,\mathbb{R})=\varnothing$. Indeed, the imaginary parts of the first kind eigenvalues of $M$ are determined by \(M\) itself, and their product cannot have different signs.

\vspace{6mm}

\noindent{\bf Global perturbation}

For any $\Phi \in \mathcal{P}(2n,\mathbb{R})$, after orthogonalization we obtain a path $\Phi^{\#}$ (see \eqref{4.4}).
In order to deal with the issue that  some end point of $\Phi^{\#}$ might be on the degenerate cycle $\mathscr{C}(2n,\mathbb{R})$ \eqref{4.5}, 
we need the following 
\begin{lemma}\label{Lemma1}
Given a $\Phi \in \mathcal{P}(2n,\mathbb{R})$, for its orthogonalization $\Phi^{\#}$, there exists a sufficiently small $\theta \geq 0$ such that both two end points of perturbed path $e^{-\theta J_n}\Phi^{\#}$ are not on the cycle $\mathscr{C}(2n,\mathbb{R})$, i.e. both $e^{-\theta J_n}\Phi^{\#}(0)$ and $e^{-\theta J_n}\Phi^{\#}(1)$ have no eigenvalues equal to $\pm{1}$. 	
\end{lemma}
\begin{proof}
	For $\Phi^{\#}(0)$ with the form as \eqref{block diagonal O}, let $0 \leq \theta_j < 2\pi$ and choose a sufficiently small $\theta_0$ such that $0<\theta_0 < \min_{\theta_j \not=0,\pi} \{\theta_j,|\pi-\theta_j|,|2\pi-\theta_j|\}$, then
	\begin{align*}
	& \det(I_{2n}-e^{-\theta_0 J_n}\Phi^{\#}(0))=\prod_{j=1}^{n}[1-\cos(\theta_j-\theta_0)]^2 + \sin^2(\theta_j-\theta_0) \not=0, \\
	& \det(-I_{2n}-e^{-\theta_0 J_n}\Phi^{\#}(0))=\prod_{j=1}^{n}[1+\cos(\theta_j-\theta_0)]^2 + \sin^2(\theta_j-\theta_0) \not=0.
	\end{align*}
	Thus, $e^{-\theta_0 J_n}\Phi^{\#}(0)$ have no eigenvalues equal to $\pm{1}$. For $e^{-\theta J_n}\Phi^{\#}(1)$, we have the similar result that there exists $\theta'_0$ such that $e^{-\theta'_0 J_n}\Phi(1)$ has no eigenvalues equal to $\pm{1}$, we choose $\theta = \min\{\theta_0,\theta'_0\}$ and  the lemma holds. In particular, if $\Phi^{\#}(0)$ and $\Phi^{\#}(1)$ have no eigenvalues equal to $\pm{1}$, we can choose $\theta=0$. This completes the proof.
\end{proof}	

Then we define an operation on $\Phi^{\#}$ as
\begin{align}\label{4.8}
	\Phi^{\#}_{\theta}(t):= e^{-\theta J_n}\Phi^{\#}(t),\ \theta \geq 0.
\end{align}
We call \eqref{4.8} the {\bf global perturbation} of $\Phi^{\#}$ with the rotation angle of $\theta$. By Lemma \ref{Lemma1}, there exists a sufficiently small $\theta \geq 0$ such that
\begin{align*}
	\Phi^{\#}_{\theta}(0), \ \  \Phi^{\#}_{\theta}(1) \notin \mathscr{C}(2n,\mathbb{R}).
\end{align*}
That is to say, orthogonalization and sufficiently small global perturbation give rise to a modified path $\Phi^{\#}_{\theta}$ whose  two endpoints are orthogonal and non-degenerate symplectic matrices. Then we show  that the rotation number of $\Phi$ also remains invariant under this operation of global perturbation.
\begin{lemma}\label{lemma orthogonalization and perturbation}
For $\Phi \in \mathcal{P}(2n,\mathbb{R})$, the rotation number \eqref{2.7}  is invariant under the operations of orthogonalizations and sufficiently small global perturbations, i.e., 
\begin{equation}\label{Delta-OP-inv}
\Delta(\Phi^{\#}_{\theta})=\Delta(\Phi).
\end{equation}
\end{lemma}

\begin{proof}
	By Lemma \ref{lemma orthogonalization}, we only need to prove $\Delta(\Phi^{\#}_{\theta})=\Delta(\Phi^{\#})$. Define
	\begin{align*}
		\gamma_0(t) = e^{-\theta t J_n}\Phi^{\#}(0),\  \  \gamma_1(t) = e^{-\theta t J_n}\Phi^{\#}(1).
	\end{align*}
	We construct a homotopic map
	\begin{align*}
		H(t,s) = \gamma_0([0,s]) \# \Phi^{\#}_{\theta s} \# (-\gamma_1([0,s]))(t), \ 0 \leq t \leq 1, \ 0 \leq s \leq 1,
	\end{align*}
	which satisfies $H(t,0)=\Phi^{\#}(t)$ and $H(t,1)=\gamma_0 \# \Phi^{\#}_{\theta} \# (-\gamma_1)(t)$. Then $\Phi^{\#}$ is homotopic to $\gamma_0 \# \Phi^{\#}_{\theta} \# (-\gamma_1)$, by \eqref{2.10} and \eqref{2.9}, we obtain
	\begin{align*}
		\Delta(\Phi^{\#}) = \Delta(\Phi^{\#}_{\theta}) + \Delta(\gamma_0) - \Delta(\gamma_1).
	\end{align*}
	Then we only need to show $\Delta(\gamma_0) = \Delta(\gamma_1)$. Given an arbitrary orthogonal symplectic matrix 
	$\begin{pmatrix}
			X & -Y \\
			Y & X
		\end{pmatrix}$, 
	Let %$\gamma(t) = e^{-\theta t J_n}O$. Then
	\begin{align*}
		\gamma(t): =e^{-\theta t J_n}\begin{pmatrix}
			X & -Y \\
			Y & X
		\end{pmatrix}=\begin{pmatrix}
		\cos(\theta t)X+\sin(\theta t)Y & \sin(\theta t)X-\cos(\theta t)Y \\
		\cos(\theta t)Y-\sin(\theta t)X & \cos(\theta t)X+\sin(\theta t)Y
		\end{pmatrix}.
	\end{align*}
By \eqref{2.5}, we have
	\begin{align*}
		\rho(\gamma(t)) & = \det((\cos(\theta t)X+\sin(\theta t))Y+i(\cos(\theta t)Y-\sin(\theta t)X)) \\
	                  	& = \det((\cos(\theta t)-i\sin(\theta t))X+(\sin(\theta t)+i\cos(\theta t))Y) \\
	                  	& = \det((\cos(\theta t)-i\sin(\theta t))(X+iY)) \\
	                  	& = e^{-in\theta t}\det(X+iY)).
	\end{align*}
Set $\det(X+iY))=e^{i\theta_0}$, then $\rho(\gamma(t))=e^{i(\theta_0-n\theta t)}$. By \eqref{2.6} and \eqref{2.7}, the rotation number $\Delta(\gamma)=-\frac{n\theta}{\pi}$. Since $\Phi^{\#}(0)$ and $\Phi^{\#}(1)$ are orthogonal symplectic matrices, then $\Delta(\gamma_0) = \Delta(\gamma_1)=-\frac{n\theta}{\pi}$. Thus, $\Delta(\Phi^{\#}) = \Delta(\Phi^{\#}_{\theta}) $.
This completes the proof of the equality \eqref{Delta-OP-inv}.
\end{proof}

\vspace{3mm}

\noindent{\bf Extension}

To obtain an integer-valued index, we then consider the extension of $\Phi^{\#}_{\theta}$. Set
\begin{align*}
	A:=\Phi^{\#}_{\theta}(0),\ \  B:=\Phi^{\#}_{\theta}(1)
\end{align*}
and then all eigenvalues of $A$ and $B$ belong to $\mathbb{S}^{1} \backslash \{\pm{1}\}$. By the orthogonalization and global perturbation, we know that $A$ and $B$ are block diagonal matrices. Let
\begin{align}\label{4.9}
	A=A_1 \diamond A_2 \diamond \cdots \diamond A_n \in Sp(2n,\mathbb{R}),
\end{align}
where $a_1 \leq a_2 \leq \cdots \leq a_n $ and
\begin{align*}
	A_j=\begin{pmatrix}
		\cos a_j & -\sin a_j \\
		\sin a_j& \cos a_j
	\end{pmatrix} \in Sp(2,\mathbb{R}), \ j=1, 2, \dots, n.
\end{align*}
has the first kind eigenvalue $\lambda_j$. Similarly, let
\begin{align}\label{4.10}
	B=B_1 \diamond B_2 \diamond \cdots \diamond B_n \in Sp(2n,\mathbb{R}),
\end{align}
where $ b_1 \leq b_2 \leq \cdots \leq b_n $ and
\begin{align*}
	B_j=\begin{pmatrix}
		\cos b_j & -\sin b_j \\
		\sin b_j& \cos b_j
	\end{pmatrix} \in Sp(2,\mathbb{R}), \ j=1, 2, \dots, n.
\end{align*}
has the first kind eigenvalue $\mu_j$. Now we want to find the end point for the extension of $\Phi^{\#}_{\theta}$.
Compared with the Conley-Zehnder index, Long index and $L_0$-index, the extension of $\Phi^{\#}_{\theta}$ has $2^n$ possible end points instead of $2$.
The variousness of the end points of the extension is the result of that the starting point of the symplectic path is a general symplectic matrix rather than the identity.
These $2^n$ possible end points can be expressed as the following form
\begin{align}\label{4.11}
	W_{A} = W_1 \diamond W_2 \diamond \cdots \diamond W_n
\end{align}
where
\begin{align*}
	W_j=\begin{pmatrix}
		\cos w_j & -\sin w_j \\
		\sin w_j& \cos w_j
	\end{pmatrix} = A_j \ \text{or} \ -A_j, \ j=1, 2, \dots, n.
\end{align*}	
Given $A=\Phi^{\#}_{\theta}(0)$, one can choose the unique one, which is determined by $B=\Phi^{\#}_{\theta}(1)$,   from the set $\{W_{A}\}$ of $2^n$ elements, denoted by $W_{A,B}$. The rule is given by
\begin{equation}\label{4.12}
	W_j = \left\{
	\begin{array}{rl}
		A_j, & Im(\lambda_j) Im(\mu_j) > 0, \\
		-A_j, & Im(\lambda_j) Im(\mu_j) < 0,
	\end{array} \right.	
\end{equation}
where $\lambda_j,\mu_j$ are the first kind eigenvalues of $A$ and $B$. According to \eqref{4.12}, the imaginary parts of the first kind eigenvalues of $W_j$ and $B_j$ have the same signs. Then the product of the imaginary parts corresponding to $W_{A,B}$ and $B$ also have the same signs and hence $W_{A,B}$ and $B$ are in the same connected components of $\mathscr{N}(2n,\mathbb{R})$ (see \eqref{4.7}), thus we can define the {\bf extension} for $\Phi^{\#}_{\theta}$ as
\begin{align}\label{4.13}
	\beta :[0,1] \to Sp(2n,\mathbb{R}) \backslash \mathscr{C}(2n,\mathbb{R}),\ \beta(0)=B,\ \beta(1)=W_{A,B}.
\end{align}

To illustrate that the  rotation number is independent of the choices of extension $\beta$, we need the following lemma which is a corollary by Lemma 1.7  of \cite{CZ2} and Lemma 3.2 of \cite{SZ}.
\begin{lemma}\label{Lemma2}
	If $\Phi$ is a loop in $Sp(2n,\mathbb{R}) \backslash \mathscr{C}(2n,\mathbb{R})$, then the rotation number of $\Phi$ (see \eqref{2.7}) is equal to zero, i.e. $\Delta(\Phi)=0$.
\end{lemma}

\begin{proof}
	Choose any loop $\Phi:[0,1] \to Sp(2n,\mathbb{R}) \backslash \mathscr{C}(2n,\mathbb{R})$, then $\Phi(0)=\Phi(1)$. Given a $M \in \Phi$, we define
	\begin{align*}
		\alpha_{j}:Sp(2n,\mathbb{R}) \backslash \mathscr{C}(2n,\mathbb{R}) \to [0,2],  \ j=1,2,...,n
	\end{align*}
	by $e^{i\pi \alpha_{j}(M)} = \frac{\lambda_j}{|\lambda_j|}$, where $\lambda_j$ are the first kind eigenvalues of $M$. We can order these first kind eigenvalues such that $\alpha_{j}(M) \leq \alpha_{j+1}(M)$. If there are no positive eigenvalues, then $\alpha_{j}(M)$ is determined by $e^{i\pi \alpha_{j}(M)} = \frac{\lambda_j}{|\lambda_j|}$ uniquely. If $\lambda_j > 0$, we choose $\alpha_{j}(M)$ such that there is the same number of $j$'s with $\alpha_{j}(M)=0$ and with $\alpha_{j}(M)=2$. We construct the function $\rho(\Phi(t))=e^{i\pi \sum_{j=1}^{n}\alpha_{j}(\Phi(t))}$ which satisfies the properties of Theorem \ref{Thm SZ} and then it is unique. Moreover, we can see that every $\alpha_{j}(\Phi(t))$ is periodic and continuous and satisfies $|\alpha_j(\Phi(1))-\alpha_j(\Phi(0))| < 1$. That is because $\lambda_j$ does not pass through $\pm{1}$ according to the condition. Then $\alpha_j(\Phi(1))-\alpha_j(\Phi(0))=0$ and hence
	\begin{align*}
		\Delta(\Phi)=\sum_{j=1}^{n} (\alpha_j(\Phi(1))-\alpha_j(\Phi(0)))=0.
	\end{align*}
	This completes the proof.
\end{proof}

For the reader's convenience of  understanding the construction of the extension and  the rule of determining $W_{A,B}$ by $A$ and $B$, we give an $ad\ hoc$ example here to show how to do this.
\begin{example}\label{Eg W_{A,B}}
We consider path of $4 \times 4$ symplectic matrices. Let $A$ and $B$ be orthogonal symplectic matrices as follows
\begin{align*}
	A=A_1 \diamond A_2, \ B=B_1 \diamond B_2,
\end{align*}
where
\begin{align*}
	A_1=\begin{pmatrix}
		\frac{\sqrt{2}}{2} & -\frac{\sqrt{2}}{2} \\
		\frac{\sqrt{2}}{2}& \frac{\sqrt{2}}{2}
	\end{pmatrix}, \
	A_2=\begin{pmatrix}
		\frac{1}{2} & -\frac{\sqrt{3}}{2} \\
		\frac{\sqrt{3}}{2} & \frac{1}{2}
	\end{pmatrix}, \
	B_1=\begin{pmatrix}
		0 & -1 \\
		1 & 0
	\end{pmatrix}, \
	B_2=\begin{pmatrix}
		0 & 1 \\
		-1 & 0
	\end{pmatrix}.
\end{align*}	
The set of eigenvalues of $A$ is $\{e^{\pm{\frac{i \pi}{4}}},e^{\pm{\frac{i \pi}{3}}}\}$ and the first kind eigenvalues are $\lambda_1=e^{\frac{i \pi}{4}},\lambda_2=e^{\frac{i \pi}{3}}$, the set of eigenvalues of $B$ is $\{e^{\pm{\frac{i \pi}{2}}},e^{\pm{\frac{3i \pi}{2}}}\}$ and the first kind eigenvalues are $\mu_1=e^{\frac{i \pi}{2}},\mu_2=e^{\frac{3i \pi}{2}}$. Since $Im(\lambda_1) Im(\mu_1)=\frac{\sqrt{2}}{2} \times \frac{\sqrt{3}}{2} > 0$	and $Im(\lambda_2) Im(\mu_2)=1 \times (-1) <0$, then we have $W_1=A_1$ and $W_2=-A_2$ by the rule \eqref{4.12}.
Hence one can obtain
\begin{align*}
	W_{A,B}=A_1 \diamond (-A_2).
\end{align*}
Then one can choose the extension as
\begin{align*}
	\beta(t)=\begin{pmatrix}
		\cos(\frac{\pi}{2}-\frac{\pi t}{4}) & -\sin(\frac{\pi}{2}-\frac{\pi t}{4})  \\
		\sin(\frac{\pi}{2}-\frac{\pi t}{4}) & \cos(\frac{\pi}{2}-\frac{\pi t}{4}) 
	\end{pmatrix} \diamond 
   \begin{pmatrix}
   \cos(\frac{3\pi}{2}-\frac{\pi t}{6}) & -\sin(\frac{3\pi}{2}-\frac{\pi t}{6})\\
   \sin(\frac{3\pi}{2}-\frac{\pi t}{6}) & \cos(\frac{3\pi}{2}-\frac{\pi t}{6})
\end{pmatrix} ,\ 0 \leq t \leq 1,
\end{align*}
which lies in the component $\mathscr{N}^{-}(2n,\mathbb{R})$ (\eqref{4.99}), satisfying $\beta(0)=B$, $\beta(1)=W_{A,B}$.
\end{example}

Now we define the Maslov-type index as
\begin{definition} \label{Def Maslov.type.index}
	For any $\Phi \in \mathcal{P}(2n,\mathbb{R})$, the Maslov-type index is defined by
	\begin{align}\label{4.14}
		\mu(\Phi) := \Delta(\Phi^{\#}_{\theta}) + \Delta(\beta).
	\end{align}	
\end{definition}
If we choose another extension $\beta{'}$, then $\beta{'} \# -\beta$ is a loop in $Sp(2n,\mathbb{R}) \backslash \mathscr{C}(2n,\mathbb{R})$. It follows from Lemma \ref{Lemma2} that
$\Delta(\beta{'} \# -\beta) =0$ and hence $\Delta(\beta{'})=\Delta(\beta)$, then $\Delta(\beta)$ is independent of the choices of $\beta$. We will show that $\mu(\Phi)$ is also independent of the choice of a sufficiently small  $\theta$ in Theorem \ref{Thm1} (2) so that it is well defined, which will be proved in section \ref{Sec-proof}.

\begin{remark}\label{method}
	The method to construct the index for general symplectic paths is not unique, we can also follow the idea from Long \cite{L0} and Liu \cite{Liu1}. If we only apply the Definition \ref{Def Maslov.type.index} to the paths starting at $I$, then for the general paths $\Phi$, we can define the index as
	\begin{equation*}
		\hat{\mu}(\Phi)=\mu(\Phi' \# \Phi) - \mu(\Phi'),
	\end{equation*}
	where $\Phi'$ is the symplectic path which starts at $I$ and ends at $\Phi(0)$, we need to show that this index is independent of the choices of $\Phi'$ so it is well defined. Since our method can deal with the general symplectic paths directly and be consistent with the one for path starting from identity, we do not have to apply such an indirect way  to accomplish the construction.	
\end{remark}

\begin{remark}[Another expression of the index $\mu$ by Cappell and Miller]\label{Sec-integ}

Recently, Cappell and Miller \cite{CM} studied our index $\mu$ and the relationship between Cappell-Lee-Miller index $\mu_{CLM}$ and $\mu$. 
In particular, they show an insightful expression of these two indices in terms of integrals plus end point corrections. Here we briefly review their expression of $\mu$, denoted by $\mu^{int}$, 
the complete proof of the equality $\mu=\mu^{int}$ refers  to  section 4 of \cite{CM}. 

Given a symplectic path $\Phi :[0,1]\to Sp(2n,\mathbb{R})$.
For each $t\in[0,1]$, let $\lambda_1(t), ...,\lambda_n(t)$ be the $n$ eigenvalues of the first kind of $\Phi(t)$. 
Recall the Salomon-Zehnder mapping $\rho:Sp(2n,\mathbb{R}) \to \mathbb{S}^{1}$ is given by  \eqref{2.1} (and \eqref{2.2}), $i.e.$
$$\rho(\Phi(t)) = \prod\limits_{j=1}^{n} \frac{\lambda_j(t)}{|\lambda_j(t) |}.$$
For each $t\in[0,1]$, we denote by $\theta_1(t), \cdots, \theta_n(t)\in [0,2\pi)$ such that $e^{i\theta_j(t)}=(\frac{\lambda_j(t)}{|\lambda_j(t) |})^2$.
Let $h(t)$ be the number of $ j $  with $[\lambda_j(\Phi(t))]^2 = +1$. Denote the canonical 1-form on $\mathbb{S}^{1}$ by $\frac{d\theta}{2\pi}$.
From $\Phi$ and $\rho$ we obtain a map $$\widetilde\Phi:=(\rho\circ\Phi)^2: [0,1]\to \mathbb{S}^{1}.$$
Then Cappell and Miller give the following expression of $\mu$:
\begin{equation}
\mu^{int}(\Phi):=\int_a^b \widetilde\Phi^*(\frac{d\theta}{2\pi})+\sum_{j=1}^n  \frac{\theta_j(0)-\theta_j(1)}{2\pi}  + \big( h(0)-h(1)\big).
\end{equation}
\end{remark}

\section{Proof of the main results}\label{Sec-proof}

In this section we will prove the main results, which shows some properties of $\mu(\Phi)$ and claims the relationships to other Maslov-type indices.

\begin{proof}[\bf Proof of Theorem \ref{Thm1}]
	We prove Theorem \ref{Thm1} (1) firstly. By \eqref{4.14}, $\mu(\Phi)$	is defined by $\Delta(\Phi^{\#}_{\theta}) + \Delta(\beta)$, let
	\begin{equation}\label{5.1}
		\Phi{'}(t) = \left\{
		\begin{array}{ll}
			\Phi^{\#}_{\theta}(2t), & 0 \leq t < \frac{1}{2},\\
			\beta(2t-1),  & \frac{1}{2} \leq t \leq 1,
		\end{array} \right.	
	\end{equation}
	then $\Phi{'}(0)=A, \Phi{'}(1)=W_{A,B}$. We can construct
	\begin{align*}	
		\rho(\Phi{'}(t))=e^{i\pi \sum_{j=1}^{n} \alpha_{j}(t)},\ j=1,2,\cdots, n,
	\end{align*}	
	where every $\alpha_{j}$ satisfies that $e^{i\pi \alpha_{j}(t)}=\frac{\lambda_j(t)}{|\lambda_j(t)|}$ and $\lambda_j(t)$ is the first kind eigenvalue of $\Phi^{'}(t)$. By \eqref{4.12}, we have $\lambda_j(0)=\pm{\lambda_j(1)}$, then $\alpha_{j}(1)-\alpha_{j}(0) \in \mathbb{Z}$ and hence
	\begin{align}\label{5.2}
		\Delta(\Phi{'})=\sum_{j=1}^{n}(\alpha_{j}(1)-\alpha_{j}(0))\in \mathbb{Z}.
	\end{align}	
	This proves that $\mu(\Phi)$ is an integer.
	
	We continue to prove Theorem \ref{Thm1} (2). By Lemma \ref{lemma orthogonalization}, we know that the rotation number is independent of the choices of the orthogonalizations.
	Suppose $\Phi^{\#}_{\theta}$ and $\Phi^{\#}_{\theta{'}}$ are two different perturbations, by Lemma \ref{lemma orthogonalization and perturbation}, we have $\Delta(\Phi^{\#}_{\theta})=\Delta(\Phi^{\#}_{\theta{'}})=\Delta(\Phi)$ . Let $\beta$ be the extension of $\Phi^{\#}_{\theta}$. In section 4, we have shown that $\Delta$ is independent of the choice of the extension. Denote the endpoint of the extension of $\Phi^{\#}_{\theta{'}}$ by $W_{A{'},B{'}}$. Then we can choose the extension of $\Phi^{\#}_{\theta{'}}$ as
	\begin{equation}\label{5.3}
		\beta{'}(t) = \left\{
		\begin{array}{ll}
			\beta'_1(3t), & 0 \leq t < \frac{1}{3}, \\
			\beta(3t-1), & \frac{1}{3} \leq t < \frac{2}{3}, \\
			\beta'_2(3t-2), & \frac{2}{3} \leq t \le 1,
		\end{array} \right.
	\end{equation}
	where $\beta'_1$ and $\beta'_2$ are the path in $Sp(2n,\mathbb{R}) \backslash \mathscr{C}(2n,\mathbb{R})$ and satisfy $\beta'_1(0)=\Phi^{\#}_{\theta'}(1)$,\\
	$\beta'_1(1)=\Phi^{\#}_{\theta}(1), \beta'_2(0) =W_{A,B}$ and $\beta'_2(1)=W_{A{'},B{'}}$. Since both $\theta$ and $\theta{'}$ are small enough, by the continuity of $\rho$ (see Theorem \ref{Thm SZ}) and \eqref{2.9}, we can see
	$|\Delta(\beta)-\Delta(\beta{'})|=|\Delta(\beta'_1)+\Delta(\beta'_2)|$ is small enough and hence
	\begin{align}\label{5.4}
		|\mu(\Phi)-\mu{'}(\Phi)| & \leq |\Delta(\Phi^{\#}_{\theta})-\Delta(\Phi^{\#}_{\theta{'}})|+|\Delta(\beta)-\Delta(\beta{'})| \\
		& = |\Delta(\beta)-\Delta(\beta{'})|
	\end{align}
	is small enough, where $\mu{'}(\Phi)=\Delta(\Phi^{\#}_{\theta{'}})+\Delta(\beta{'})$. Then we obtain $\mu(\Phi)=\mu{'}(\Phi)$ because they are integer. This implies that $\mu(\Phi)$ is also independent of the sufficiently small $\theta$ and hence Definition \ref{Def Maslov.type.index} is well defined.
	
	If $\Phi, \Psi$ are homotopic with fixed end points, by Theorem \ref{Thm1} (2), we can choose the same $\theta$ and $\beta$ such that $\Phi^{\#}_{\theta} \# \beta$ and $\Psi^{\#}_{\theta} \# \beta$ are homotopic with fixed end points, by \eqref{2.10}, we have $\Delta(\Phi^{\#}_{\theta} \# \beta)=\Delta(\Psi^{\#}_{\theta} \# \beta)$, that is $\mu(\Phi)=\mu(\Psi)$ and so Theorem \ref{Thm1} property (3) holds.
	
	For $0<a<1$, we can choose a suitable perturbation such that $\Phi^{\#}_{\theta}(a) \notin \mathscr{C}(2n,\mathbb{R})$ and set $C=\Phi^{\#}_{\theta}(a)$. The end points for extensions of $\Phi^{\#}_{\theta}, \Phi^{\#}_{\theta}([0,a])$ and $\Phi^{\#}_{\theta}([a,1])$ are $W_{A,B},W_{C,B}$ and $W_{A,C}$. Denote these extensions by $\beta, \beta_3$ and $\beta_4$, by \eqref{2.9} and \eqref{4.14}, we have
	\begin{align}\label{5.5}
		\mu(\Phi([0,a]))+\mu(\Phi([a,1]))=\Delta(\Phi^{\#}_{\theta})+\Delta(\beta_3)+\Delta(\beta_4).
	\end{align}
	The first kind eigenvalues of $\beta_4$ is from $\lambda_{C}$ to $\lambda_{W}$ on $\mathbb{S}^{1} \backslash \{\pm{1}\}$ and here $\lambda_{C} \in \sigma(C),\ \lambda_{W} \in \sigma(W_{A,C})$. We consider the path $-\beta_3 \# \beta$ which starts at $W_{C,B}$ and ends at $W_{A,B}$, the first eigenvalues change along $\mathbb{S}^{1} \backslash \{\pm{1}\}$ from $\pm  \lambda_{C}$ to $\pm \lambda_{W}$. Then we obtain $\Delta(\beta_4)=\Delta(-\beta_3 \# \beta)$ and hence $\Delta(\beta_3)+\Delta(\beta_4)=\Delta(\beta)$, so
	\begin{align}\label{5.6}
		\mu(\Phi([0,a]))+\mu(\Phi([a,1]))=\Delta(\Phi^{\#}_{\theta})+\Delta(\beta)=\mu(\Phi) .
	\end{align}
	This proved Theorem \ref{Thm1} (4).
	
	By \eqref{2.12}, Theorem \ref{Thm1} (5) is obvious. We only need to choose the extension of the form as
	\begin{align*}
		\beta(z_1,z_2)=(\beta_5z_1, \beta_6z_2),
	\end{align*}
	where $\beta_5, \beta_6$ are the extensions of $(\Phi_1)^{\#}_{\theta}, (\Phi_2)^{\#}_{\theta}$ and this property holds. 
	
	Since $T \Phi T^{-1}$ and $\Phi$ have the same first kind eigenvalues for any $t \in [0,1]$. After the orthogonalizations, by \eqref{4.3} and Lemma \ref{lemma orthogonalization}, we obtain that they have the same endpoints and the same rotation numbers. We can choose the same perturbations and extensions because the Maslov-type index $\mu$ is idependent of choices of them. Hence we have $\mu(T \Phi T^{-1}) = \mu(\Phi)$ and then Theorem \ref{Thm1} (6) holds. This completes the proof.
\end{proof}

Then we verify the relationship between our Maslov-type index and other indices. 
$\mu(\Phi)$ is actually an intersection number of path $\Phi$ and the cycle $\mathscr{C}(2n,\mathbb{R})$ (see \eqref{4.5}), 
which is essentially determined by every resulted path of  first kind eigenvalues of $\Phi(t)$. 
The differences between $\mu(\Phi)$ and other  indices are determined by different methods of various constructions.
We will show the details and prove Theorems \ref{Thm2}, \ref{Thm3}, \ref{Thm4}, \ref{Thm5} and \ref{Thm6}.

\begin{proof}[\bf Proof of Theorem \ref{Thm2}]
{\bf(i)} Non-degenerate case. By \eqref{3.3}, Lemma \ref{lemma orthogonalization and perturbation} and \eqref{4.14}, we have
\begin{align*}
\mu_{CZ}(\Phi)&= \Delta(\Phi) + \Delta(\gamma)= \Delta(\Phi^{\#}_{\theta}) + \Delta(\gamma),\\
\mu(\Phi) &= \Delta(\Phi^{\#}_{\theta}) + \Delta(\beta).
\end{align*}
Then the difference between $\mu(\Phi)$ and $\mu_{CZ}(\Phi)$ depends on the different extensions. Recall \eqref{2.2}, \eqref{2.6}  and \eqref{2.7}, we can see that the contribution  to either index of a corresponding extension path of matrices can be attributed to the contribution of each resulted extending path of first kind eigenvalues.
Since non-degenerate path $\Phi$ satisfies $\Phi(0)=I, \det (I-\Phi(1)) \not = 0$, then  every resulted path of the every first kind eigenvalues $\lambda(t)$ starts at $\lambda(0)=1$ 
and ends at $\lambda(1)\neq 1$. 
We just need to study contrastively the contributions of each resulted extending path of first kind eigenvalues for $\mu(\Phi)$ and $\mu_{CZ}(\Phi)$.
We do not have to consider those conjugate pair of first kind eigenvalues with $| \lambda |<1$ and $\mathrm{Im} \lambda \neq 0$, 
since their contributions to rotation number are always cancelled by each other.
All remaining cases about the  first kind eigenvalues of $\Phi(1)$ are as follows: 
	
(1) \ First kind eigenvalue $\lambda(1) \in \R$.  For $\mu_{CZ}(\Phi)$, the extension path $\gamma$ (see \eqref{ext-gamma}) results in extending $\lambda(1)$ to $\frac{1}{2}$ or $-1$  
and the corresponding extending path of the first kind eigenvalue does not crossing $1$. Then the contribution of extension $\gamma$ to $\mu_{CZ}(\Phi)$ attributed to $\lambda(1)$  is equal to zero. 
For $\mu(\Phi)$, after the orthogonalization and global perturbation, $\lambda|_{\Phi_{\theta}^{\#}(1)} = e^{-i\theta}$ when $\lambda(1)>0$ 
or $\lambda|_{\Phi_{\theta}^{\#}(1)} = e^{i(\pi - \theta)}$ when $\lambda(1)<0$. By \eqref{4.12} and \eqref{4.13}, the extension path $\beta$ results in that
the terminal point of extending path of this first kind eigenvalue (along the unit circle) is $e^{-i\theta}$ when $\lambda(1)>0$ or is $e^{i(\pi - \theta)}$ when $\lambda(1)<0$. 
Hence the contribution to $\mu(\Phi)$ attributed to $\lambda(1)$  is also equal to zero.

(2) \ First kind eigenvalue $\lambda(1) \in \mathbb{S}^{1}$ and $\mathrm{Im} \lambda(1) > 0$. 
For $\mu_{CZ}(\Phi)$, from \eqref{ext-gamma} we see that every path of the first kind eigenvalues resulted from the extending path $\gamma$  starts at $\lambda(1)$ then going along the unit circle anti-clockwise and ends at $-1$. 
For $\mu(\Phi)$, since $\lambda(0)=1$, after the global perturbation, we have $\mathrm{Im}( e^{-i\theta}\lambda(0))= \mathrm{Im}( e^{-i\theta})< 0$. 
According to the condition $\mathrm{Im} \lambda(1) > 0$ and $\theta$ is sufficiently small, we see $\mathrm{Im}( e^{-i\theta}\lambda(1)) > 0$. 
By \eqref{4.12} and \eqref{4.13}, each  path of the first kind eigenvalues resulted from the extending path  $\beta$ starts at $e^{-i\theta}\lambda(1)$ 
then going along the unit circle  anti-clockwise and ends at $-e^{-i\theta}$. 
The rotation angle along the unit circle from $\lambda(1)$ to $-1$ and the one from  $e^{-i\theta}\lambda(1)$ to  $-e^{-i\theta}$ are the same.
Then we can see that the contributions to the rotation numbers of the extending paths $\gamma$ and $\beta$ in two cases attributed to the corresponding paths of  first kind eigenvalues are the same.
	
(3) \ First kind eigenvalue $\lambda(1) \in \mathbb{S}^{1}$ and $\mathrm{Im} \lambda(1) < 0$. 
The argument is similar to case (2). For $\mu_{CZ}(\Phi)$, every path of the first kind eigenvalues resulted from the extending path $\gamma$  starts at $\lambda(1)$ then going along the unit circle while clockwise (since $\mathrm{Im} \lambda(1) < 0$ and the extending part of $\lambda(t)$ is not permitted to pass across 1) and ends at $-1$.
For $\mu(\Phi)$, the difference with case (2) is that  $\mathrm{Im}( e^{-i\theta}\lambda(1)) <0$. By \eqref{4.12} and \eqref{4.13}, we will take the different terminal point for $\beta$ such that 
each  path of the first kind eigenvalues resulted from the extending path  $\beta$ starts at $e^{-i\theta}\lambda(1)$ 
then going along the unit circle  anti-clockwise and ends at $e^{-i\theta}\lambda(0)=e^{-i\theta}$.
Then we can see that the contribution to the rotation number of the extending paths $\gamma$ attributed to the corresponding path of  first kind eigenvalues is one less than that of  $\beta$.

In summary, the crucial  difference between values of $\mu_{CZ}(\Phi)$ and  $\mu(\Phi)$ is caused by the case (3) above.	
We denote by
	\begin{align}\label{5.7}
		r(M):= \#\{\lambda \ |\  \lambda \in \sigma_{1}(M) \cap \mathbb{S}^{1}  \ \mathrm{and}\  \mathrm{Im} \lambda < 0 \},
	\end{align}
where the counting involves the multiplicity of the first kind eigenvalues. 
So each first kind eigenvalue of $\Phi(1)$ on $\mathbb{S}^{1}$ with $\mathrm{Im} \lambda(1) < 0$ leads to contribution to the value of $\mu$  one more  than that of $\mu_{CZ}$,  while the contributions to both indices caused by  eigenvalues in other cases are the same. Then we have
	\begin{align*}
		\mu_{CZ}(\Phi)=\mu_{L}(\Phi)=\mu(\Phi)-r(\Phi(1)).
	\end{align*}
	
\noindent{\bf(ii)} Now we consider the case of degenerate path $\Phi \in \mathcal{P}(2n,\mathbb{R})$, i.e. $\Phi(0)=I, \det (I-\Phi(1)) = 0$. 
We only need to further consider the new issue about the first kind eigenvalues $1$ of $\Phi(1)$ and its variation caused by the operation of rotational perturbation. By \eqref{3.3}, \eqref{def-muL} and Lemma \ref{lemma orthogonalization and perturbation} , for sufficiently small $s>0$,
$$\mu_{L}(\Phi) = \mu_{CZ}(\Phi(-s,\cdot))= \Delta(\Phi(-s,\cdot)) + \Delta(\gamma_s),$$ %= \Delta(\Phi(-s,\cdot)^{\#}_\theta) + \Delta(\gamma_s), $$ 
where $\gamma_s$ is the Conley-Zehnder's extension path starting at $\Phi(-s,1)$. 
%$\Phi(-s,\cdot)^{\#}_\theta$ is the modified path after our operations of orthogonalization and global perturbation of $\Phi(-s,\cdot)$. 
Recall \eqref{rotational perturbation} and \eqref{2.9}, we have
$$ \Delta(\Phi(-s,\cdot)) =  \Delta(\Phi(-s,\cdot)|_{[0,t_0]}) + \Delta(\Phi(-s,\cdot)|_{[t_0,1]}) =  \Delta(\Phi|_{[0,t_0]}) + \Delta(\Phi(-s,\cdot)|_{[t_0,1]}) ,$$
and $\Phi(-s, t)$ is sufficiently close to  $\Phi(1)$   for $\forall\ t \in [t_0,1]$. Thus,
$$\mu_{L}(\Phi) =  \Delta(\Phi|_{[0,t_0]}) + \Delta(\Phi(-s,\cdot)|_{[t_0,1]} + \Delta(\gamma_s).$$ 
On the other hand, by Lemma \ref{lemma orthogonalization and perturbation} and \eqref{2.9},
$$\mu(\Phi) = \Delta(\Phi^{\#}_{\theta}) + \Delta(\beta)= \Delta(\Phi) + \Delta(\beta)=  \Delta(\Phi|_{[0,t_0]}) + \Delta(\Phi|_{[t_0,1]}) + \Delta(\beta).$$
Recall that $\Delta(\gamma_s)$ depends only on $\Phi(-s,1)$.
That is to say, $\Delta(\gamma_s)$ is determined by the first kind eigenvalues of $\Phi(-s,1)$. 
Recall Proposition \ref{standard form}, there exists $P \in Sp(2n,\R)$ such that
	\begin{align*}
		P \Phi(1) P^{-1} = N_{k_1}(\bb_1) \diamond N_{k_2}(\bb_2) \diamond \cdots \diamond N_{k_q}(\bb_q) \diamond M_0.
	\end{align*} 
The rotational perturbation   slightly changes one of those first kind eigenvalues $1$ of $\Phi(1)$ to $\lambda'(1)\neq 1$, which is one of the first kind eigenvalues of $\Phi(-s,1)$. 
	
Considering our index $\mu(\Phi)$,  since $\Phi(0)=I_{2n}$ then every $\lambda(0)=1$. 
Then for each first kind eigenvalue $\lambda(1)=1$ of $\Phi(1)$, the resulted contribution to $\Delta(\beta)$ is equal to zero. 

Then considering Long index $\mu_L(\Phi)$, we study all cases of  the contribution to $\Delta(\gamma_s)$ resulted from the first kind eigenvalue $\lambda'(1)$, which are as follows: 
\smallskip
	
\noindent (1) \ If $\lambda'(1) \in \mathbb{S}^{1}$ and $Im(\lambda'(1)) > 0$, the path of the first kind eigenvalues resulted from the extending path $\gamma_s$  starts at $\lambda'(1)$ then going along the unit circle anti-clockwise and ends at $-1$. 
Since $\lambda'(1)$ is close to 1, then the resulted contribution to  $\Delta(\gamma_s)$  is almost equal to $1$. Note that $\Delta(\Phi(-s,\cdot)|_{[t_0,1]})$ is close to $\Delta(\Phi|_{[t_0,1]})$ and both indices are integers, which implies that the resulted contribution to $\mu(\Phi)$ is exactly one less than that to $\mu_{L}(\Phi)=\mu_{CZ}(\Phi)(-s,\cdot)$.\smallskip

\noindent (2) \ If $\lambda'(1) \in \mathbb{S}^{1}$ and $\lambda'(1) \in (0,1)$, the path of the first kind eigenvalues resulted from the extending path $\gamma_s$  starts at $\lambda'(1)$ then going along the positive real axis and ends at $\frac1 2$. 
Then the resulted contribution to  $\Delta(\gamma_s)$  is equal to zero, which means the resulted contribution to $\mu(\Phi)$ is exactly equal to that to $\mu_{L}(\Phi)=\mu_{CZ}(\Phi)(-s,\cdot)$.\smallskip

\noindent (3) \ If $\lambda'(1) \in \mathbb{S}^{1}$ and $Im(\lambda'(1)) < 0$, the path of the first kind eigenvalues resulted from the extending path $\gamma_s$  starts at $\lambda'(1)$ then going along the unit circle clockwise and ends at $-1$. 
Since $\lambda'(1)$ is close to 1, then the resulted contribution to  $\Delta(\gamma_s)$  is almost equal to $-1$. By argument similar to the one in case (1) above,
one can see that the resulted contribution to $\mu(\Phi)$ is one more than that to  $\mu_{L}(\Phi)=\mu_{CZ}(\Phi)(-s,\cdot)$.\smallskip

For each $j=1,\cdots, q$, we define
\begin{align*}
& l_{j,1}(M) := \# \{\lambda\ |\ \lambda\in\sigma_1\big( N_{k_j}(\bb_j)  e^{-\theta J_{k_j}} \big)\cap  \mathbb{S}^1\ \mathrm{and}\ \ \mathrm{Im}\ \lambda >0\} ,\\
& l_{j,2} (M):= \# \{\lambda\ |\ \lambda\in\sigma_1\big(N_{k_j}(\bb_j)  e^{-\theta J_{k_j}} \big)\cap (0,1)\}, \\
& l_{j,3}(M) := \# \{\lambda\ |\ \lambda\in\sigma_1\big(N_{k_j}(\bb_j)  e^{-\theta J_{k_j}} \big)\cap  \mathbb{S}^1 \ \mathrm{and}\ \ \mathrm{Im}\ \lambda<0\},
\end{align*} 
where $\theta = s \theta_0 >0$ and $s,\ \theta_0$ are given by \eqref{rotational perturbation}. 
For sufficiently small $\theta$, the definitions above are independent of the choices of 
$\theta$. 
Since each $\lambda'(1)$ is perturbed from eigenvalue 1 of some normal form $N_{k_j}(\bb)$, one can see that each $N_{k_j}(\bb)$ contributes  to the value of $\mu(\Phi)$ more than to that of $\mu_{L}(\Phi)$ by the amount of $l_{j,3}(\Phi(1))-l_{j,1}(\Phi(1))$. For a symplectic matrix $M$ possessing eigenvalue 1, denote by 
	\begin{align}\label{l}
		l(M)=\sum_{j=1}^{q} (l_{j,3}(M)-l_{j,1}(M)).
	\end{align} 
If $M$ does not possess eigenvalue 1, let $l(M)=0$.
Then all normal forms totally contribute  to $\mu(\Phi)$ more than to $\mu_{L}(\Phi)$ by the amount $l(\Phi(1))$. As for other first kind eigenvalues($\not=1$) of $\Phi(1)$, the argument is the same as the one for Conley-Zehnder index in {\bf(i)} above. Then the difference of $\mu$ and $\mu_{L}$ is given by
	\begin{align*}
		\mu_{L}(\Phi)=\mu(\Phi)-r(\Phi(1))-l(\Phi(1)).
	\end{align*}
	This has proved \eqref{1.4} and  the proof of Theorem \ref{Thm2} is complete.
	\end{proof}

\bigskip
	
	\begin{proof}[\bf Proof of Theorem \ref{Thm3}]
	Theorem \ref{Thm3} is the corollary of \eqref{1.4}, \eqref{3.9} and \eqref{3.10}. If $\Phi \in \mathcal{P}(2n,\mathbb{R})$ satisfies $\Phi(0)=I$, then
	\begin{align*}
		i_{L_0}(\Phi)=\mu_{L}(\Phi) - c(\Phi(1))=\mu(\Phi)-r(\Phi(1))-l(\Phi(1))-c(\Phi(1)),
	\end{align*}
	where $c(\Phi(1))$ is the $L_0$-concavity given by \eqref{3.10}. This has proved \eqref{1.6}.
\end{proof}

\bigskip

\begin{proof}[\bf Proof of Theorem \ref{Thm4}]
	If $\Phi$ is a general path, by \eqref{1.4}, \eqref{Long SPS} and Theorem \ref{Thm1} (4), we obtain
	\begin{align*}
		\hat{\mu}_{L}(\Phi)&=\mu_{L}(\Phi' \# \Phi)-\mu_{L}(\Phi') \\
		&=\mu(\Phi' \# \Phi)-r(\Phi(1))-l(\Phi(1))-(\mu(\Phi')-r(\Phi(0))-l(\Phi(0))) \\
		&=\mu(\Phi' \# \Phi)-\mu(\Phi')+r(\Phi(0))-r(\Phi(1))+l(\Phi(0))-l(\Phi(1)) \\
		&=\mu(\Phi)+r(\Phi(0))-r(\Phi(1))+l(\Phi(0))-l(\Phi(1)),
	\end{align*}
	where $\Phi'$ is the symplectic path which starts at $I$ and ends at $\Phi(0)$.
	By \eqref{1.6} and  Theorem \ref{Thm1} (4), we can obtain that
	\begin{align*}
		\hat{i}_{L_0}(\Phi)&=i_{L_0}(\Phi'' \# \Phi)-i_{L_0}(\Phi'') \\
		&=\mu(\Phi'' \# \Phi)-r(\Phi(1))-l(\Phi(1))-c(\Phi(1))\\
		& \ \ \ -(\mu(\Phi'')-r(\Phi(0))-l(\Phi(0))-c(\Phi(0))) \\
		&=\mu(\Phi'' \# \Phi)-\mu(\Phi'')\\
		& \ \ \ +r(\Phi(0))-r(\Phi(1))+l(\Phi(0))-l(\Phi(1))+c(\Phi(0))-c(\Phi(1)) \\
		&=\mu(\Phi)+r(\Phi(0))-r(\Phi(1))+l(\Phi(0))-l(\Phi(1))+c(\Phi(0))-c(\Phi(1)),
	\end{align*}
	where $\Phi''$ is the symplectic path which starts at $I$ and ends at $\Phi(0)$. This completes the proof of Theorem \ref{Thm4}.
\end{proof}

\bigskip

\begin{proof}[\bf Proof of Theorem \ref{Thm5}]
	Next, we prove Theorem \ref{Thm5}. Let $f(t)=(L_1,L_2(t))=(L_1, \Phi(t) L_1)$ be a pair of Lagrangian paths and $L_1=\mathbb{R}^{n} \times \{0\}$, by Proposition \ref{sp to osp},
	then there exists an orthogonal symplectic path $O$ such that $L_2(t)$ crosses $\Sigma(L_1)$  if and only $O(t)$ crosses $\mathscr{C}(2n,\mathbb{R})$ (see \eqref{4.5}).
	We will compute the index of $f(t)$ by a proper path $\bar{f}(t)=f_1 \# f \# f_2 (t)$ (see \eqref{bar f}). Set $d=\dim(L_1(0) \cap L_2(0))=0$. If $e^{-\theta J_n}L_2$ crosses $\Sigma(L_1)$ transversally at all intersection time, then $\mu_{CLM}(\bar{f})= p-q$ by the review of Section \ref{subsec-CLM}. It follows from \eqref{3.12} that we know
	\begin{align*}
		\mu_{CLM}(f)= d+p-q.
	\end{align*}
	Let $\bar{O}$, $O_1$ and $O_2$ denote the corresponding orthogonal symplectic path of $\bar{f}$, $f_1$ and $f_2$. For an intersection time $t_j, \bar{O}(t_j) \in \mathscr{C}(2n,\mathbb{R})$. $e^{-\theta J_n}L_2(t)$ crosses $\Sigma(L_1)$ transversally  if and only $\bar{O}(t)$ crosses $\mathscr{C}(2n,\mathbb{R})$. If $\delta>0$ is small enough, by Theorem \ref{Thm1} (4), we have
	\begin{align*}
		\mu(\bar{O})=\mu(\bar{O}([0,t_j-\delta]))+\mu(\bar{O}([t_j-\delta,t_j+\delta]))+\mu(\bar{O}([t_j+\delta,1])).
	\end{align*}
	If one of first kind eigenvalues $\lambda(t)$ of $\bar{O}(t)$ crosses $\pm{1}$ in $(t_j-\delta, t_j-\delta)$ anti-clockwise, then the contribution to $\mu(\bar{O}([t_j-\delta, t_j+\delta]))$ is $+1$. It is equal to $-1$ if $\lambda(t)$ crosses in the opposite direction. Thus, the first kind eigenvalue $\lambda(t)$ crossing $\pm{1}$ in $(t_j-\delta, t_j-\delta)$ will contribute $\pm{1}$ to $\mu(\bar{O})$. The number of all contribution are equal to $p+q$ but $p$ intersection points contribute $p$ and $q$ intersection points contribute $-q$. Then we have $\mu_{CLM}(\bar{f})=\mu(\bar{O})$ and hence
	\begin{align*}
		\mu_{CLM}(f)=d+ \mu(\bar{O})=d+ \mu(O_1)+ \mu(O) + \mu(O_2).
	\end{align*}
By Definition \eqref{Def Maslov.type.index}, we have $\mu(O_1)=-d$ and $\mu(O_2)=0$, thus $\mu_{CLM}(f)=\mu(O)$.
\end{proof}

\smallskip

\begin{proof}[\bf Proof of Theorem \ref{Thm6}]
	Finally, we prove Theorem \ref{Thm6}. By Proposition \ref{RS:sp to osp}, there exists a corresponding orthogonal symplectic path as
	$$
	O(t)=\begin{pmatrix}
		\cos \Theta(t) & -\sin \Theta(t) \\
		\sin \Theta(t) & \cos \Theta(t)
	\end{pmatrix},
	$$
	where
	$$\Theta(t)={\rm diag} \{\theta_1(t), \theta_2(t), \cdots, \theta_n(t)\},\ \theta_j(0) \in [0,\pi), \ 1 \leq j \leq n,$$
	such that
	$$\mu_{RS}(\Phi)=\mu_{RS}(O).$$
 Choose the Lagrangian frame (see \eqref{3.13}) as
	\begin{align*}
		Z(t)=	\begin{pmatrix}
			\cos\Theta(t) \\
			\sin\Theta(t)
		\end{pmatrix} ,\
	\end{align*}
	then the crossing form (see \eqref{3.15}) is
	\begin{align*}
		\Gamma(OL_2,L_2,t)(v)=\langle \cos\Theta(t)u, [\sin\Theta(t)]{'}u \rangle- \langle \sin\Theta(t)u, [\cos\Theta(t)]{'}u \rangle = \Theta{'}(t)u^{2} ,
	\end{align*}
	where $L_2=\{0\} \times \R^n, v=(0,u) \in O(t)L_2 \cap L_2$. When $t \not=0,1$, then $\Gamma(OL_2,L_2,t)(v)$ contributes $+1$ to $\mu_{RS}(O)$ if $\theta{'}_j(t)>0$ at the intersection time and this also means that $\lambda(t):=e^{i\theta_j(t)}$ of $O(t)$ crosses $\pm{1}$ anti-clockwise. The circumstance of the case $\theta{'}_j(t)<0$ is similar. Then $\mu(O)$ and $\mu_{RS}(O)$ is equal without considering the two endpoints. When $t=0,1$, there are four cases about the endpoints of $O$ as
	\begin{align*}
		&(1)\ \theta{'}_j(0)>0,\ \text{the contribution to} \ (\mu(O), \mu_{RS}(O)) \ \text{is} \ (1,\frac{1}{2}), \\
		&(2)\  \theta{'}_j(0)<0,\ \text{the contribution to} \ (\mu(O), \mu_{RS}(O)) \ \text{is} \ (-\frac{1}{2},0), \\
		&(3)\  \theta{'}_j(1)>0,\ \text{the contribution to} \ (\mu(O), \mu_{RS}(O)) \ \text{is} \ (\frac{1}{2},0), \\
		&(4)\  \theta{'}_j(1)<0,\ \text{the contribution to} \ (\mu(O), \mu_{RS}(O)) \ \text{is} \ (-\frac{1}{2},-1).
	\end{align*}
	Thus, if $\theta{'}_j(t) \not=0$, then the contribution to $\mu(O)$ is $\displaystyle \frac{1}{2}$
	more than $\mu_{RS}(O)$ when $t=0$ and  $\displaystyle \frac{1}{2}$
	less than $\mu_{RS}(O)$ when $t=1$. Define
	\begin{align}\label{5.8}
		s(t)=  \text{rank} \Gamma(OL_2,L_2,t)
	\end{align}
	be the rank of the crossing form. Then we have
	\begin{align*}
		\mu_{RS}(\Phi)=\mu_{RS}(O)=\mu(O)+\frac{1}{2}(s(0)-s(1)).
	\end{align*}
 By Proposition \ref{Prop LG=RS}, we have
 \begin{align*}
 	\mu_{LG}(\Phi)=\mu_{RS}(\Phi)=\mu(O)+\frac{1}{2}(s(0)-s(1)).
 \end{align*}
This completes the proof of Theorem \ref{Thm6}.
\end{proof}
Finally, we give two concrete examples  to show the relationships of these indices.
\begin{example} \label{Eg2}
	Let
	\begin{align*}
		\Phi(t) = \begin{pmatrix}
			\cos\frac{3\pi t}{2} & -\sin\frac{3\pi t}{2} \\
			\sin\frac{3\pi t}{2} & \cos\frac{3\pi t}{2} \\
		\end{pmatrix},\ 0 \leq t \leq 1.
	\end{align*}
	By Example \ref{Eg1}, we have $\mu_{CZ}(\Phi)=1$. By \eqref{4.4}, we consider the perturbation $e^{-\theta J}, 0<\theta <\frac{1}{2}$, then
	\begin{align*}
		\Phi^{\#}_{\theta}(t) = \begin{pmatrix}
			\cos(\frac{3\pi t}{2}-\theta) & -\sin(\frac{3\pi t}{2}-\theta) \\
			\sin(\frac{3\pi t}{2}-\theta) & \cos(\frac{3\pi t}{2}-\theta) \\
		\end{pmatrix},\ 0 \leq t \leq 1.
	\end{align*}
	Since $\rho(\Phi_{\theta}(0))=e^{-i\theta}$ and $\rho(\Phi_{\theta}(1))=e^{i(\frac{3\pi}{2}-\theta)}$, then $Im(\Phi_{\theta}(0))Im(\Phi_{\theta}(1))>0$. By \eqref{4.12}, the end point of the extension of $\Phi_{\theta}$ is $\Phi_{\theta}(0)$. Choose this extension as
	\begin{align*}
		\beta(t) = \begin{pmatrix}
			\cos(\frac{(t+3)\pi}{2}-\theta) & -\sin(\frac{(t+3)\pi}{2}-\theta) \\
			\sin(\frac{(t+3)\pi}{2}-\theta) & \cos(\frac{(t+3)\pi}{2}-\theta) \\
		\end{pmatrix},
	\end{align*}
	then $\mu(\Phi)=\Delta(\Phi)+\Delta(\beta)=\frac{3}{2}+\frac{1}{2}=2$ and hence
	\begin{align*}
		\mu_{CZ}(\Phi)=\mu(\Phi)-1,
	\end{align*}
	where $1=r(\Phi(1))$ (see \eqref{5.7}) because  $Im(e^{\frac{3\pi i}{2}})<0$.  Choose the Lagrangian frame (see \eqref{3.13}) of $\Phi U$ as
$Z(t)=\begin{pmatrix}
			\cos\frac{3\pi t}{2} \\
			\sin\frac{3\pi t}{2}
		\end{pmatrix}$, then the crossing form (see \eqref{3.15}) is
\begin{align*}
\Gamma(\Phi U,U,t)(v)=\langle \cos\frac{3\pi t}{2}u, [\sin\frac{3\pi t}{2}]^{'}u \rangle- \langle \sin\frac{3\pi t}{2}u, [\cos\frac{3\pi t}{2}]^{'}u \rangle =  \frac{3\pi}{2}u^{2},
\end{align*}
	where $U=\{0\} \times \mathbb{R}, v=(0,u) \in \Phi U \cap U$. At the intersection time $t=0$ and $t=\frac{2}{3}, u \not= 0$, then $\text{Sign}\Gamma(\Phi U,U,t)=1$ and we have
	\begin{align*}
		\mu_{RS}(\Phi)=\frac{1}{2}\text{Sign}\Gamma(\Phi U,U,0)+\text{Sign}\Gamma(\Phi U,U,\frac{2}{3})=\frac{3}{2}.
	\end{align*}
	Moreover, we can see $s(0)=\text{rank}\Gamma(\Phi U,U,0)=1$ and $s(1)=0$, then
	\begin{align*}
		\mu_{RS}(\Phi)=\mu(O)+\frac{1}{2}(s(0)-s(1)), O=\Phi.
	\end{align*}
	Let $\Psi = \Phi_{\theta},\ L_1= \mathbb{R} \times \{0\},\ L_2(t)=\Psi(t) L_1$, then $\dim(L_1(0) \cap L_2(0))=0$. The intersection time of $L_2(t)$ crossing $\Sigma(L_1)=L_1$ is $t=\theta$ and $t=\theta+\frac{1}{2}$, both two crossings are anti-clockwise, by \eqref{3.12} we obtain
	\begin{align*}
		\mu_{CLM}(\Psi)=2=\mu(\Psi).
	\end{align*}
\end{example}
\begin{example}\label{Eg degenerate paths}
	We consider the degenerate path \eqref{1.3}, $i.e.$
	\begin{align*}
		\Phi(t) = \begin{pmatrix}
			1 & 0 \\
			0 & 1 
		\end{pmatrix} \diamond
	\begin{pmatrix}
		1 & -t \\
		0 & 1
	\end{pmatrix}
	,\ 0 \leq t \leq 1.
	\end{align*}
	This path is degenerate and $L_0$-degenerate. For our definition, $\Phi^{\#}(1)=I_4$, and $\Phi^{\#}$ is homotopic to the constant path $I_4$ with fixed endpoints, then $\mu(\Phi)=\mu(\Phi^{\#})=\mu(I_4)=0$ because all eigenvalues of $I_4$ are invariant.
	
	For Long index, $\mu_{L}(\Phi)=\mu_{L}(I_2)+\mu_{L}(\Phi')=\mu_{L}(\Phi')-1$, where
	\begin{align*}
		\Phi'(t) = \begin{pmatrix}
			1 & -t\\
			0 & 1\\
		\end{pmatrix},\ 0 \leq t \leq 1.
	\end{align*}
	the rotational perturbation will change the first kind eigenvalues of $\Phi'(1)$ to $\{\lambda \ | \ \lambda>0 \}$, 
	by Definition \ref{Def L}, we obtain $\mu_{L}(\Phi')=0$ and hence $\mu_{L}=-1$. 
	Here the normal forms are $I_2$ and $\Phi'(1)$, then $l_{1,3}(\Phi(1))=1$ and all the rest are zero. So $l(\Phi(1))=1$. Note that $r(\Phi(1))=0$. 
	Thus,  the equality $\mu_{L}(\Phi)=\mu(\Phi)-r(\Phi(1))-l(\Phi(1))$ holds for the path \eqref{1.3}. 
	
	For $L_0$-index, $i_{L_0}(\Phi)=i_{L_0}(I_2)+i_{L_0}(\Phi')=i_{L_0}(\Phi')-1$, 
	$\Phi'$ is a $L_0$-nondegenerate path, by \eqref{3.7}, we obtain $i_{L_0}(\Phi')=0$, then $i_{L_0}(\Phi)=-1$.
	
	Thus, we see that our index $\mu$ is different from Long index $\mu_{L}$ and $L_0$-index $i_{L_0}$.
\end{example}

\bigskip

\end{document}